\documentclass[a4paper, 10pt]{amsart}

\usepackage[all,cmtip]{xy}

\usepackage{fullpage}

\usepackage{tikz-cd}
\usetikzlibrary{arrows}
\usepackage{caption}
\usepackage{ stmaryrd }

\tikzset{
commutative diagrams/.cd,
arrow style=tikz,
diagrams={>=latex}}

\usepackage[utf8]{inputenc}
\usepackage{amsmath}
\usepackage{amsthm}
\usepackage{ mathrsfs }
\usepackage{ amssymb }
\usepackage{mathtools}
\usepackage{enumitem}

\usepackage{tikz}

\theoremstyle{definition}
\newtheorem{definition}{Definition}[section]
\newtheorem{proposition}[definition]{Proposition}
\newtheorem{corollary}[definition]{Corollary}

\newtheorem{lemma}[definition]{Lemma}
\newtheorem{remark}[definition]{Remark}
\newtheorem{theorem}[definition]{Theorem}
\newtheorem*{theorem*}{Theorem}

\newtheorem{example}{Example}[section]

\newcommand{\mf}{\mathcal{F}}

\newcommand{\id}{\mathrm{Id}}

\newcommand{\ad}{\mathrm{ad}}

\usepackage {tikz}
\usetikzlibrary {positioning}

\newcommand{\End}{\mathrm{End}}
\newcommand{\As}{\mathsf{As}}
\newcommand{\KAs}{\As^{\mbox{!`}}}
\newcommand{\frakg}{\mathfrak{g}}
\newcommand{\fraka}{\mathfrak{a}}

\newcommand{\MC}{\mathsf{MC}}
\newcommand{\dgr}{\mathsf{dgr}}
\newcommand{\dgraphs}{\mathsf{dgraphs}}

\newcommand{\coder}{\mathrm{Coder}}
\newcommand{\codiff}{\mathrm{Codiff}}

\newcommand{\Hom}{\mathrm{Hom}}
\newcommand{\bch}{\mathrm{bch}}

\newcommand{\Aut}{\mathrm{Aut}}
\newcommand{\Duf}{\mathrm{Duf}}
\newcommand{\Sym}{\mathrm{Sym}}
\newcommand{\dgraphsuni}{\mathsf{dgraphs}^{\mathrm{uni}}}

\title{On a homotopy version of the Duflo isomorphism}
\author{Matteo Felder}
\begin{document}

\begin{abstract}
For a finite dimensional Lie algebra $\mathfrak{g}$, the Duflo map $S\mathfrak{g}\rightarrow U\frakg$ defines an isomorphism of $\mathfrak{g}$-modules. On $\mathfrak{g}$-invariant elements it gives an isomorphism of algebras. Moreover, it induces an isomorphism of algebras on the level of Lie algebra cohomology $H(\mathfrak{g},S\mathfrak{g})\rightarrow H(\mathfrak{g}, U\mathfrak{g})$. However, as shown by J. Alm and S. Merkulov, it cannot be extended in a universal way to an $A_\infty$-isomorphism between the corresponding Chevalley-Eilenberg complexes. In this paper, we give an elementary and self-contained proof of this fact using a version of M. Kontsevich's graph complex.
\end{abstract}

\address{%
Matteo Felder\\
Dept. of Mathematics\\
University of Geneva\\
2-4 rue du Li\`evre\\
1211 Geneva 4\\
Switzerland\\            
Matteo.Felder@unige.ch 
}
\maketitle
\tableofcontents

\section*{Introduction}

For a finite dimensional Lie algebra $\frakg$, the Duflo map $\Duf:S\frakg\rightarrow U\frakg$ is an isomorphism of $\frakg$-modules. It is defined as the composition of the symmetrization map with the Duflo element, the formal power series on $\frakg$,
\begin{equation*}
\det\left(\frac{e^{\ad_x/2}-e^{-\ad_x/2}}{\ad_x}\right)^{1/2},
\end{equation*}
viewed as a differential operator of infinite order acting on $S\frakg$. It is a non-trivial fact that when restricted to $\frakg$-invariant elements, the map $\Duf:S\frakg^\frakg\rightarrow U\frakg^\frakg$ is an isomorphism of algebras \cite{Duflo77}. Moreover, it induces an isomorphism of algebras $H(\frakg,S\frakg)\rightarrow H(\frakg,U\frakg)$ on the level of Lie algebra cohomology (\cite{Shoikhet2003}, \cite{PevznerTorossian2004}). In particular, this implies that on chains $\Duf:C(\frakg,S\frakg) \rightarrow C(\frakg,U\frakg)$ respects the algebra structures up to homotopy. More precisely, there exists a map $\Duf_2:C(\frakg,S\frakg)^{\otimes 2}\rightarrow C(\frakg,U\frakg)$ which measures the failure of the Duflo map to be an algebra morphism, i.e. it satifies,
\begin{equation*}
\Duf (m_{C(\frakg,S\frakg)})=m_{C(\frakg,U\frakg)} (\Duf\otimes \Duf)+ d ( \Duf_2)+\Duf_2(d\otimes 1+1\otimes d).
\end{equation*}
It is natural to require $\Duf_2$ to satisfy further compatibility conditions up to homotopy. However, it turns out that this procedure may not be extended (in a universal way) to higher homotopies of arbitrary order. In other words, the Duflo map does not extend to an $A_\infty$-isomorphism. This was shown by J. Alm in \cite{Alm2011},\cite{Alm2014}, and later again in collaboration with S. Merkulov \cite{AlmMerkulov}. They work with a variety of techniques from the theory of graph complexes, and use deep results by T. Willwacher \cite{Willwacher2014}. The purpose of this text is to present a self-contained, elementary and (hopefully) more accessible proof of this fact using yet another variant of M. Kontsevich's graph complex.

Denote by $\tilde m_\Duf$ the product on $C(\frakg,S\frakg)$ defined by pulling back the product on $C(\frakg,S\frakg)$ via the Duflo map. J. Alm and S. Merkulov's theorem (\cite{Alm2014}, Proposition 5.3.0.10) may then be reformulated as follows.
\begin{theorem*}
There does not exist a universal (i.e. independent of the specific choice of the Lie algebra $\frakg$) $A_\infty$-isomorphism,
\begin{equation*}
f:(C(\frakg,S\frakg),d_{C(\frakg,S\frakg)}, m_{C(\frakg,S\frakg)})\rightarrow (C(\frakg,S\frakg),d_{C(\frakg,S\frakg)},\tilde m_\Duf),
\end{equation*}
whose first component $f_1$ is the identity.
\end{theorem*}

The strategy of our proof is the following. We begin by introducing a version of M. Kontsevich's graph complex \cite{Kontsevich2003}. For $n\geq 1$, we denote by $\dgraphs(n)$ the graded vector space spanned by directed graphs having $n$ ``external" vertices (labeled, of arbitrary valence) and possibly some ``internal" vertices (unlabeled, at most trivalent) modulo a set of relations encoding the Lie algebra structure of $\frakg$. The product $\prod_{n\geq 1} \dgraphs(n)$ carries the structure of a graded Lie algebra. Next, we recall a variant of M. Kontsevich's representation $B:\dgraphs(n)\rightarrow \Hom(C(\frakg,S\frakg)^{\otimes n},C(\frakg,S\frakg))$. It turns out that both $m_{C(\frakg,S\frakg)}$ and $\tilde m_\Duf\in \Hom(C(\frakg,S\frakg)^{\otimes 2},C(\frakg,S\frakg))$ may be described via this representation. In particular, $\tilde m_\Duf$ corresponds precisely to M. Kontsevich's celebrated star product on $C^\infty (\frakg^*)$, i.e. smooth functions on the dual Lie algebra $\frakg^*$, viewed as a Poisson manifold (\cite{Kontsevich2003},\cite{Shoiket2001}). Within this setting, we find that \emph{universal} $A_\infty$-structures on $C(\frakg,S\frakg)$ are in bijection with Maurer-Cartan elements of the graded Lie algebra $\prod_{n\geq 1}\dgraphs(n)$, that is, (linear combinations of) graphs $\alpha$ of degree one, satisfying the equation,
\begin{equation*}
[\alpha,\alpha]=0.
\end{equation*}
Moreover, $A_\infty$-isomorphic structures correspond to gauge equivalent Maurer-Cartan elements. This reduces the question of the existence of such an $A_\infty$-isomorphism to a combinatorial problem in graph theory. If $\alpha_0$ and $\alpha_\Duf$ are the Maurer-Cartan elements inducing the two products on $C(\frakg,S\frakg)$, then for them to be gauge equivalent means that there is a $\xi\in \prod_{n\geq 1} \dgraphs(n)$ of degree zero such that,
\begin{equation*}
\alpha_0=\alpha_\Duf+[\xi,\alpha_\Duf]+\frac{1}{2!}[\xi,[\xi,\alpha_\Duf]]+\dots=e^{[\xi,-]}\alpha_\Duf.
\end{equation*}
The graphs $\alpha_0$ and $\alpha_\Duf$ are given explicitly up to a certain order. We may thus define $\xi$ inductively in such a way that it deforms $\alpha_\Duf$ into $\alpha_0$. In this process, we find that the graph,
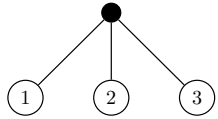
\begin{center}
\resizebox{!}{1.5cm}{
\begin {tikzpicture}[-latex, auto ,node distance =1cm and 1cm ,on grid ,
semithick];

\tikzstyle{vertex1}=[circle,minimum size=2mm,draw=black,fill=black];
\tikzstyle{vertex2}=[circle,minimum size=2mm,draw=black,fill=white];
\tikzset{edge/.style = {-}};

\node[vertex2] (5) {$1$};
\node[vertex2] (6) [right=of 5,xshift=0.5cm] {$2$};
\node[vertex2] (7) [right=of 6,xshift=0.5cm] {$3$};

\node[vertex1] (8) [above=of 6,yshift=0.5cm] {};

\draw[edge] (5) to [] (8);
\draw[edge] (8) to [] (6);
\draw[edge] (8) to [] (7);
\end{tikzpicture}
}
\end{center}
\noindent defines an obstruction class which cannot be forced to vanish by gauge transformations.

\subsection*{Acknowledgements}
I am very grateful to Anton Alekseev and Thomas Willwacher for numerous useful discussions and suggestions. I also thank Ricardo Campos and Florian Naef for many fruitful exchanges. This work was supported by the grant MODFLAT of the European Research Council (ERC).

\section{Preliminaries}\label{section:preliminaries}

\subsection{$A_\infty$-structures}

The material presented in this section is well-known. We follow the textbook by J.-L. Loday and B. Vallette \cite{lodayvallette2012} and B. Keller's exposition \cite{Keller2001}. We work over a field $\mathbb{K}$ of characteristic zero. Let $A$ be a graded vector space. The \emph{suspension} $sA$ of $A$ is defined via the degree shift $(sA)_{p}=A_{p+1}$. Recall also the \emph{suspension map} $s:A\rightarrow sA$, the canonical map of degree $-1$ which sends $a$ to $a$. This sign convention corresponds to the one in \cite{Keller2001}, and \emph{not} to the one in \cite{lodayvallette2012}. 

\begin{definition}
An $A_\infty$-structure on $A$ is a collection of maps $m=\{m_n:A^{\otimes n} \rightarrow A\}_{n\geq 1}$ of degree $2-n$ satisfying for $n\geq 1$,
\begin{equation}\label{eq:ainfty}
\sum (-1)^{r+st} m_u (1^{\otimes r}\otimes m_s\otimes 1^{\otimes t})=0,
\end{equation}
where the sum runs over all decompositions $n=r+s+t$ and $u:=r+1+t$.
\end{definition}

\begin{definition}
For $(A,m)$, $(B,l)$ two $A_\infty$-algebras, a morphism of $A_\infty$-algebras $f:A\rightarrow B$ is a collection of maps $\{f_n:A^{\otimes n}\rightarrow B\}_{n\geq 1}$ of degree $1-n$ satisfying for $n\geq 1$,
\begin{equation}\label{eq:ainftymap}
\sum (-1)^{r+st} f_u (1^{\otimes r} \otimes m_s \otimes 1^{\otimes t})=\sum (-1)^s l_r (f_{i_1}\otimes f_{i_2}\otimes \cdots \otimes f_{i_r}),
\end{equation}
where the first sum runs over all decompositions $n=r+s+t$ and $u:=r+1+t$, and the second sum runs over all $1\leq r\leq n$ and all decompositions $n=i_1+\cdots +i_r$. Also,
\begin{equation*}
s=(r-1)(i_1-1)+(r-2)(i_2-1)+\cdots +2(i_{r-2}-1)+(i_{r-1}-1).
\end{equation*}
Given two $A_\infty$-morphisms $f:A\rightarrow B$ and $g:B\rightarrow C$, the $n$-th component of their composition is defined via the identity,
\begin{equation*}
(f\circ g)_n=\sum (-1)^s f_r (g_{i_1}\otimes \cdots \otimes g_{i_r}),
\end{equation*}
where the summation runs over the same elements as in equation \eqref{eq:ainftymap}.
\end{definition}

\begin{remark}
We adopt the notation as in the two previous definitions. Note that $m_1$ is of degree one and satisfies $m_1^2=0$, and therefore defines a differential on $A$. Moreover, as $f_1m_1=l_1 f_1$, $f_1$ defines a chain map between the complexes $(A,m_1)$ and $(B,l_1)$.
\end{remark}

\begin{definition}
A morphism of $A_\infty$-algebras $f=\{f_n\}_{n\geq 1}$ is called \emph{isomorphism} (\emph{quasi-isomorphism}) if $f_1$ is an isomorphism (quasi-isomorphism). An $A_\infty$-algebra $(A,m)$ is called \emph{strict} if $m_n=0$ for all $n\geq 3$.
\end{definition}

\subsection{$A_\infty$-structures via the convolution Lie algebra}
Consider the endomorphims operad of $A$, $\End_A=\{\Hom(A^{\otimes n},A)\}_{n\geq 1}$. For $f\in \Hom(A^{\otimes n},A)$ and $g\in \Hom(A^{\otimes m},A)$, the partial composition $f\circ_i g\in \Hom(A^{\otimes n+m-1},A)$ is simply,
\begin{equation*}
(f\circ_i g)(a_1\otimes \cdots \otimes a_{n+m-1}):=f(a_1\otimes \cdots \otimes a_{i-1}\otimes g(a_{i}\otimes \cdots \otimes a_{i+m})\otimes a_{i+m+1}\otimes \cdots \otimes a_{n+m-1}).
\end{equation*}

Next, consider the non-symmetric operad $\As$. It is one-dimensional in each arity, $\As(n):=\mathbb{K} \mu_n$, where the generators $\mu_n$ are of degree zero. The partial composition is given by $\mu_n \circ_{i} \mu_m=\mu_{n+m-1}$ for any $i$. The Koszul dual cooperad of $\As$, denoted $\KAs(n)$, is also one-dimensional in each arity, but with generators $\mu_n^c$ of degree $1-n$ (\cite{lodayvallette2012}, Section 9.1.5). The \emph{convolution Lie algebra} is,
\begin{equation*}
\frakg_{\As,A}:=\Hom(\KAs,\End_A)=\prod\limits_{n\geq 1} \Hom(\KAs(n),\End_A(n)).
\end{equation*}
We equip it with the pre-Lie product $\star$ defined for $F,G\in \Hom(\KAs,\End_A)$ by,
\begin{equation*}
F\star G:= \gamma_{(1)}\circ (F\circ_{(1)} G) \circ \Delta_{(1)},
\end{equation*}
where $\circ_{(1)}$ denotes the infinitesimal composite, $\gamma_{(1)}: \End_A\circ_{(1)} \End_A\rightarrow \End_A$ the infinitesimal composition map and $\Delta_{(1)}:\KAs\rightarrow \KAs\circ_{(1)}\KAs$ the infinitesimal decomposition map (\cite{lodayvallette2012}, Section 6.4.2). We consider $\frakg_{\As,A}$ as differential graded Lie algebra with zero differential. As graded vector spaces, we identify,
\begin{equation*}
\frakg_{\As,A}=\prod\limits_{n\geq 1} \Hom(\KAs(n),\End_A(n))\cong \prod\limits_{n\geq 1} \End_A(n)\otimes {\KAs}^*(n)=\prod\limits_{n\geq 1} \End_A(n)\otimes \mathbb{K}{\mu_n^c}^*,
\end{equation*}
where ${\mu_n^c}^*$ is dual to $\mu_n^c$, and therefore of degree $n-1$. The element $F\in \Hom(\KAs(n),\End_A(n))$ which sends $\mu_n^c$ to $f\in \End_A(n)$ will be mapped to $f\otimes {\mu_n^c}^*$ by the bijection above. Next, recall that,
\begin{equation*}
\prod\limits_{n\geq 1} \End_{sA}(n)=\prod\limits_{n\geq 1} \Hom((sA)^{\otimes n},sA),
\end{equation*}
endowed with the operation, 
\begin{equation*}
\tilde f\star' \tilde g:=\sum\limits_{i=1}^n \tilde f \circ_i \tilde g,
\end{equation*}
for $\tilde f\in \End_{sA}(n)$ and $\tilde g\in \End_{sA}(m)$ also defines a pre-Lie algebra (\cite{lodayvallette2012}, Section, 5.9.15). Moreover, note that the commutative diagram,
\begin{center}
\begin{tikzcd}
   (sA)^{\otimes n} \arrow{r}{\tilde f}  
   & sA  \\ 
  A^{\otimes n} \arrow{r}{f}\arrow{u}{s^{\otimes n}}
   & A\arrow{u}{s}.
\end{tikzcd}
\end{center}
defines a bijection between $\End_{sA}(n)$ and $\End_A(n)$. This correspondence yields the identification,
\begin{equation*}
\prod\limits_{n\geq 1} \End_{sA}(n)=\prod\limits_{n\geq 1} \Hom((sA)^{\otimes n},sA)\cong \prod\limits_{n\geq 1} \End_A(n)\otimes \mathbb{K}{\mu_n^c}^*,
\end{equation*}
by mapping $\tilde f \in \End_{sA}(n)$ to $f\otimes {\mu_n^c}^*\in \End_A(n)\otimes \mathbb{K}{\mu_n^c}^*$. We then have the following result.

\begin{proposition}\label{prop:prelieisos}
We endow the space $\prod\limits_{n\geq 1} \End_A(n)\otimes \mathbb{K}{\mu_n^c}^*$ with the operation,
\begin{equation*}
(f\otimes {\mu_n^c}^*)\star'' (g \otimes {\mu_m^c}^*):=\sum\limits_{i=1}^n (-1)^{|g|(1-n)+(i+1)(m-1)} (f\circ_i g)\otimes {\mu_{n+m-1}^c}^* .
\end{equation*}
Then $\star''$ defines a pre-Lie product and the isomorphisms of graded vector spaces defined above are isomorphisms of pre-Lie algebras, i.e.
\begin{equation*}
(\frakg_{\As,A},\star ) \cong (\prod\limits_{n\geq 1} \End_A(n)\otimes \mathbb{K}{\mu_n^c}^*,\star'') \cong (\prod\limits_{n\geq 1} \End_{sA}(n), \star ').
\end{equation*}
\end{proposition}

\begin{proof}
See \cite{lodayvallette2012}, Proposition 10.1.16.
\end{proof}

We recall one more identification. Let $V$ be a graded vector space and $\overline{T}V=\prod_{n\geq 1}V^{\otimes n}$ the reduced (completed) tensor algebra. Equipped with the usual deconcatenation coproduct $\Delta$ defined by,
\begin{equation*}
\Delta(v_1\otimes \cdots \otimes v_n)= \sum\limits_{i=1}^{n-1} (v_1\otimes \cdots \otimes v_i)\otimes (v_{i+1}\otimes \cdots\otimes v_{n}),
\end{equation*}
$\overline{T}V$ describes a graded coalgebra. Let now $W$ denote another graded vector space. A linear map $f:(\overline{T}V,\Delta_V)\rightarrow (\overline{T}W,\Delta_W)$ is called a \emph{coalgebra morphism} if $\Delta_W \circ f=(f\otimes f)\circ \Delta_V$. A \emph{coderivation} $b$ of $\overline{T}V$ is a linear map $b:\overline{T}V\rightarrow\overline{T}V$ such that $\Delta\circ b=(b\otimes 1+1\otimes b)\circ \Delta$. The graded commutator of two coderivations is again a coderivation, and we denote by $\coder(\overline T V)$ the Lie algebra of coderivations of $\overline T V$. It is a well-known result that
there is a natural bijection between coderivations of $\overline T V$ and collections of linear maps $\prod_{n\geq 1} \Hom(V^{\otimes n},V)$ (see, for instance, \cite{lodayvallette2012}, Proposition 6.3.7.). More precisely, given a coderivation $b:\overline T V\rightarrow \overline T V$, its component mapping $V^{\otimes n}$ to $V^{\otimes u}$ is given by,
\begin{equation*}
\sum\limits_{\substack{r+s+t=n \\ u=r+1+t}} 1^{\otimes r}\otimes b_s \otimes 1^{\otimes t},
\end{equation*}
where $b_s=proj_V \circ b :V^{\otimes s}\rightarrow V$ is the projection of $b$ on the space of cogenerators. Note also that the composition $a\circ b$ of two coderivations of $\overline T V$ yields the following family of maps in $\prod_{n\geq 1}\Hom(V^{\otimes n},V)$,
\begin{equation*}
(a\circ b)_n=\sum\limits_{\substack{r+s+t=n \\ u=r+1+t}} a_u(1^{\otimes r}\otimes b_s \otimes 1^{\otimes t})=\sum\limits_{\substack{n+1=u+s\\ 1\leq i \leq u}} a_u (1^{\otimes i-1}\otimes b_s\otimes 1^{\otimes u-i})=\sum\limits_{\substack{n+1=u+s\\ 1\leq i \leq u}} {a_u\circ_i b_s}= (a \star' b)_n.
\end{equation*}
This gives an isomorphism of Lie algebras,
\begin{equation}\label{eq:coder}
\coder(\overline T V)\cong \prod\limits_{n\geq 1}\End_V(n).
\end{equation}

\begin{definition}
An element $\alpha\in \frakg_{\As,A}$ is called \emph{Maurer-Cartan element} if it is of degree one and satisfies the \emph{Maurer-Cartan equation},
\begin{equation*}
\alpha\star \alpha=0.
\end{equation*}
We denote the set of Maurer-Cartan elements of $\frakg_{\As,A}$ by $\MC(\frakg_{\As,A})$.
\end{definition}

\begin{remark}
By Proposition \ref{prop:prelieisos}, the data of a Maurer-Cartan element $\alpha\in \MC(\frakg_{\As,A})$ is equivalent to a collection $b=(b_n)_{n\geq 1}\in \prod_{n\geq 1} \End_{sA}(n)$ of degree one elements satisfying $b\star' b=0$. Equivalently, $b$ describes a degree one coderivation of $\overline T sA$ which squares to zero. This in turn translates to,
\begin{equation*}
(b\star' b)_n=(b\circ b)_n=\sum\limits_{\substack{n=r+s+t \\u:=r+1+t}} b_u (1^{\otimes r}\otimes b_s\otimes 1^{\otimes t})=0
\end{equation*}
for all $n\geq 1$. Moreover, the commutative diagram above applied to any $b_n$ gives a map $m_n:=s^{-1}b_n s^{\otimes n}:A^{\otimes n}\rightarrow A$ of degree $|m_n|=-n+1-(-1)=2-n$, and the collection $(m_n)_{n\geq 1}$ satisfies precisely the set of equations \eqref{eq:ainfty}, thus defining an $A_\infty$-structure on $A$. Denoting by $\codiff(\overline T sA)$ the set of degree one coderivations squaring to zero (also called \emph{codifferentials}), the reasoning above gives two bijections,
\begin{equation*}
\MC(\frakg_{\As,A})\cong \codiff(\overline T sA)\cong A_\infty - \text{structures on }A.
\end{equation*}
\end{remark}

\begin{remark}\label{Rmk:coalgebra}
Assume that $(\overline T sA,a)$ and $(\overline T sB,b)$ are differential coalgebras (i.e. coalgebras equipped with a codifferential). A similar reasoning as above implies that a differential coalgebra map $\varphi:\overline T sA\rightarrow \overline T sB$ (i.e. a coalgebra map that commutes with the codifferentials) of degree zero gives a family of degree zero maps $\{\varphi_n:(sA)^{\otimes n}\rightarrow sB\}_{n\geq 1}$ satisfying,
\begin{equation}\label{eq:coalgmap}
\sum \varphi_u (1^{\otimes r} \otimes {a}_s \otimes 1^{\otimes t})=\sum  {b}_r (\varphi_{i_1}\otimes \varphi_{i_2}\otimes \cdots \otimes \varphi_{i_r}),
\end{equation}
where the summations run over the same elements as in equation \eqref{eq:ainftymap}. Setting $f_n:=s^{-1}\varphi_n s^{\otimes n}$, $m_n:=s^{-1}a_n s^{\otimes n}$ and $l_n:=s^{-1} b_n s^{\otimes n}$, equation \eqref{eq:coalgmap} becomes exactly equation \eqref{eq:ainftymap}, that is, the collection $f=\{f_n\}_{n\geq 1}$ describes an $A_\infty$-morphism $f:(A,\{m_n\}_{n\geq 1})\rightarrow (B, \{l_n\}_{n\geq 1})$. This yields a bijection between differential coalgebra maps $(\overline T sA,a)\rightarrow (\overline T sB,b)$ of degree zero and $A_\infty$-morphisms $A\rightarrow B$, equipped with the corresponding $A_\infty$-structures.
\end{remark}

\begin{remark}\label{rmk:autisom}
The previous remark shows that the group of degree zero differential coalgebra automorphisms of $\overline T s A$, denoted $\Aut^0(\overline T sA)$ acts on $\codiff(\overline T sA)$ by the adjoint action, that is, $\varphi. b := \varphi \circ b \circ \varphi^{-1}$ for $b\in \codiff(\overline TsA)$ and $\varphi \in \Aut^0(\overline T sA)$. Moreover, both $b$ and $\varphi.b$ give $A_\infty$-structures on $A$ (say $m$ and $m_\varphi$). The family $\{\varphi_n:(sA)^{\otimes n}\rightarrow sA\}_{n\geq 1}$ obtained by the automorphism $\varphi$ induces an $A_\infty$-isomorphism $f:(A,m)\rightarrow (A,m_\varphi)$. We thus have a bijection,
\begin{equation*}
\codiff(\overline T sA)/\Aut^0(\overline TsA)\cong \{A_\infty-\text{structures on }A\}/\{A_\infty-\text{isomorphisms}\}.
\end{equation*}
\end{remark}

\begin{remark}
We denote the Lie subalgebra of $\coder(\overline TsA)$ given by degree zero coderivations by $\coder^0(\overline TsA)$. It acts on degree one coderivations via the adjoint action. Furthermore, it corresponds to the Lie algebra of the automorphism group $\Aut^0(\overline TsA)$ and therefore acts on the set of codifferentials via the formula,
\begin{equation*}
c.b:=e^{\ad_c}(b)=e^c\circ b\circ e^{-c},
\end{equation*}
for $c\in \coder^0(\overline TsA)$, $b\in \codiff(\overline TsA)$ and $\ad_c(-)=[c,-]$.
\end{remark}

\subsection{Gauge equivalences}\label{subsection:gaugeequivalence}
A clear exposition of the material below can be found in W. M. Goldman and J. J. Millson's paper \cite{Goldman1988}. We follow the more concise Appendix B of \cite{Dolgushev2012}. 
The graded Lie algebra $\frakg_{\As,A}$ has a natural descending filtration given by,
\begin{equation*}
\mf^p\frakg_{\As,A}:=\prod\limits_{n\geq p+1} \Hom(\KAs(n),\End_A(n)).
\end{equation*}
This filtration is complete and compatible with the Lie bracket, that is,
\begin{equation*}
\frakg_{\As,A}=\varprojlim\frakg_{\As,A}/\mf^p\frakg_{\As,A} \text{ and } [\mf^p\frakg_{\As,A},\mf^q\frakg_{\As,A}]\subset \mf^{p+q}\frakg_{\As,A}.
\end{equation*}
Moreover, $\mf^0\frakg_{\As,A}=\frakg_{\As,A}$. The degree zero elements $\frakg^0_{\As,A}$ form a Lie subalgebra of $\frakg_{\As,A}$ and the completeness condition ensures that $\frakg^0_{\As,A}$ is a pro-nilpotent Lie algebra. It may thus be exponentiated to the pro-unipotent group $\exp(\frakg^0_{\As,A})$ which consists of the set $\frakg^0_{\As,A}$ equipped with the Baker-Campbell-Hausdorff  product $\bch$. It acts on the set of degree one elements $\frakg^1_{\As,A}$ via the formula,
\begin{equation*}
\xi.\alpha:=\exp (\ad_{\xi})\alpha=\alpha+\ad_\xi (\alpha)+\frac{1}{2!} \ad^2_\xi (\alpha)+\dots
\end{equation*}
for $\xi \in \exp(\frakg^0_{\As,A})$ and $\alpha \in \frakg^1_{\As,A}$. Again, completeness allows us to make sense of the series above. Notice also that this is indeed a group action since $\exp(\ad_{\xi_1})\exp(\ad_{\xi_2})=\exp(\ad_{\bch(\xi_1,\xi_2)})$. We then have the following well-known result.

\begin{lemma}
The action of $\exp(\frakg^0_{\As,A})$ on $\frakg^1_{\As,A}$ preserves the set of Maurer-Cartan elements $\MC(\frakg_{\As,A})$.
\end{lemma}

\begin{proof}
We refer to Section 1 of \cite{Goldman1988}.
\end{proof}

\begin{definition}
The action of $\exp(\frakg^0_{\As,A})$ defines an equivalence relation on the set of Maurer-Cartan elements $\MC(\frakg_{\As,A})$. We say that two Maurer-Cartan elements $\alpha_1$, $\alpha_2\in \MC(\frakg_{\As,A})$ are \emph{gauge equivalent} if there is a $\xi\in \exp(\frakg^0_{\As,A})$ such that $\alpha_2=\xi. \alpha_1$.
\end{definition}

\begin{proposition}\label{prop:mcgauge}
Gauge equivalent Maurer-Cartan elements of $\frakg_{\As,A}$ correspond bijectively to $A_\infty$-isomorphic $A_\infty$-structures on $A$. We thus have the following bijections,
\begin{equation*}
\MC(\frakg_{\As,A})/\exp(\frakg^0_{\As,A})\cong\codiff(\overline T sA)/\Aut^0(\overline TsA)\cong \{A_\infty\text{-structures on }A\}/\{A_\infty\text{-isomorphisms}\}.
\end{equation*}
\end{proposition}

\begin{proof}
The identifications in Proposition \ref{prop:prelieisos} and the isomorphism \eqref{eq:coder} give an isomorphism of Lie algebras,
\begin{equation*}
\frakg^0_{\As,A}\cong \coder^0(\overline TsA),
\end{equation*}
which preserves the respective adjoint actions on $\frakg^1_{\As,A}\cong \coder^1(\overline TsA)$. Moreover, this allows us to identify $\exp(\frakg^0_{\As,A})$ with the group of automorphisms $\Aut^0(\overline TsA)$ such that the respective actions on $\MC(\frakg_{\As,A})\cong \codiff(\overline TsA)$ coincide. Thus,
\begin{equation*}
\MC(\frakg_{\As,A})/\exp(\frakg^0_{\As,A})\cong\codiff(\overline T sA)/\Aut^0(\overline TsA).
\end{equation*}
which, together with Remark \ref{rmk:autisom}, implies the statement.
\end{proof}

We finish this section by recalling a useful technical lemma.

\begin{lemma}\label{lemma:technical}
Let $\fraka$ be a graded Lie algebra. Assume that there is a second positive grading on $\fraka$ compatible with the Lie algebra structure, i.e.
\begin{equation*}
\fraka=\bigoplus\limits_{i\geq 1}\fraka^{(i)} \text{ and } [\fraka^{(i)},\fraka^{(j)}]\subset \fraka^{(i+j)}.
\end{equation*}
Let $\alpha\in \MC(\fraka)$ be a Maurer-Cartan element of $\fraka$. Decompose $\alpha$ with respect to the second grading, that is,
\begin{equation*}
\alpha=\alpha^{(1)}+\alpha^{(k)}+\alpha^{(k+1)}+\dots
\end{equation*}
where $\alpha^{(i)}\in \fraka^{(i)}$ and $k\geq 2$. Then $\alpha^{(1)}\in \MC(\fraka)$, the bracket $[\alpha^{(1)},-]$ defines a differential on $\fraka$ and if the cohomology class $[\alpha^{(k)}]\neq 0 \in H^1(\fraka,[\alpha^{(1)},-])$, the Maurer-Cartan elements $\alpha$ and $\alpha^{(1)}$ are not gauge equivalent.
\end{lemma}

\begin{proof}
The low order expansion of the Maurer-Cartan equation for $\alpha$ reads,
\begin{equation*}
0=\frac{1}{2}[\alpha,\alpha]=\frac{1}{2}[\alpha^{(1)},\alpha^{(1)}]+[\alpha^{(1)},\alpha^{(k)}]+\dots.
\end{equation*}
Since $[\alpha^{(1)},\alpha^{(1)}]$ is the only contribution to $\alpha^{(2)}$, it must equal zero. Thus $\alpha^{(1)}\in \MC(\fraka)$. Together with the Jacobi identity this implies $[\alpha^{(1)},[\alpha^{(1)},-]]=0$, i.e. bracketing with $\alpha^{(1)}$ defines a differential on $\fraka$. Next, assume $[\alpha^{(k)}]\neq 0 \in H^1(\fraka,[\alpha^{(1)},-])$ and that $\alpha$ and $\alpha^{(1)}$ are gauge equivalent Maurer-Cartan elements, that is, there exists $b\in \fraka$ of degree zero such that $\alpha=e^{\ad(b)} \alpha^{(1)}$. Expanding this equation with respect to the second grading yields inductively,
\begin{equation*}
[b^{(j)},\alpha^{(1)}]=0,
\end{equation*}
for $1\leq j\leq k-2$, and $[b^{(k-1)},\alpha^{(1)}]=\alpha^{(k)}$. But then $[\alpha^{(k)}]=0 \in H^1(\fraka,[\alpha^{(1)},-])$, leading to a contradiction. Therefore, $\alpha$ and $\alpha^{(1)}$ cannot be gauge equivalent.
\end{proof}

\subsection{The Duflo isomorphism}\label{section:duflo}
We follow D. Calaque and C. A. Rossi's lecture series \cite{RossiCalaque}. Let $\frakg$ be a finite dimensional Lie algebra. It acts on the symmetric algebra $S\frakg$ and the universal enveloping algebra $U\frakg$ by the adjoint action. Consider the formal power series on $\frakg$ given by the (modified) Duflo element,
\begin{equation*}
J(x):=\det\left(\cfrac{e^{\ad_x/2}-e^{-\ad_x/2}}{\ad_x}\right).
\end{equation*}
Moreover, recall the symmetrization map,
\begin{equation*}
\Sym:S\frakg \rightarrow U\frakg,\hspace{1cm} v_1\cdots v_n\mapsto \frac{1}{n!}\sum\limits_{\sigma \in S_n}v_{\sigma(1)}\cdots  v_{\sigma(n)}.
\end{equation*}
It is an isomorphism of filtered vector spaces, but not an algebra isomorphism (the product on $S\frakg$ being commutative, while the one on $U\frakg$ is not unless $\frakg$ is abelian). M. Duflo's theorem \cite{Duflo77} states that the composition,
\begin{equation*}
\Duf:=\Sym\circ J^{1/2}:(S\frakg)^{\frakg} \rightarrow (U\frakg)^{\frakg}
\end{equation*}
defines an algebra isomorphism on $\frakg$-invariant elements. Here, we identify $x\in \frakg$ with the vector field $\frac{\partial}{\partial x^*}$. It acts by derivation on $S\frakg$, which may be viewed as polynomial functions on $\frakg^*$. In this way, we may view the formal power series $J^{1/2}(x)$ on $\frakg$ as an infinite-order differential operator $J^{1/2}$ on $\frakg^*$.

By pulling back the product $m_{U\frakg}$ on $U\frakg$ to $S\frakg$ via the Duflo isomorphism, we obtain a second associative product,
\begin{equation*}
m_{\Duf}:=\Duf^{-1}\circ m_{U\frakg} \circ \Duf^{\otimes 2}:S\frakg^{\otimes 2}\rightarrow S\frakg
\end{equation*}
on the symmetric algebra. On invariant elements $m_{\Duf}$ coincides with the usual commutative product $m_0$.

Next, consider the Chevalley-Eilenberg complexes, $(C^\bullet(\frakg,S\frakg)=S\frakg\otimes \bigwedge^\bullet \frakg^*,d_{S\frakg\otimes \wedge \frakg^*})$ and $(C^\bullet (\frakg,U\frakg)=U\frakg\otimes \bigwedge^\bullet \frakg^*,d_{U\frakg\otimes \wedge \frakg^*})$ \cite{Chevalley1948}. Results by B. Shoikhet \cite{Shoikhet2003} and M. Pevzner and C. Torossian \cite{PevznerTorossian2004} show that the map $\Duf\otimes \id:S\frakg\otimes \bigwedge \frakg^*\rightarrow U\frakg\otimes \bigwedge \frakg^*$ induces an isomorphism of algebras on the level of cohomology. Moreover, this allows us to define two associative products on $S\frakg\otimes \bigwedge \frakg^*$, namely,
\begin{align*}
\tilde m_0:&(S\frakg\otimes \bigwedge \frakg^*)^{\otimes 2}\cong (S\frakg)^{\otimes 2} \otimes (\bigwedge \frakg^*)^{\otimes 2}\overset{m_0\otimes m_{\wedge\frakg^*}}{\longrightarrow} S\frakg\otimes \bigwedge \frakg^*\\
\tilde m_\Duf:&(S\frakg\otimes \bigwedge \frakg^*)^{\otimes 2}\cong (S\frakg)^{\otimes 2} \otimes (\bigwedge \frakg^*)^{\otimes 2}\overset{m_\Duf\otimes m_{\wedge\frakg^*}}{\longrightarrow} S\frakg\otimes \bigwedge \frakg^*
\end{align*}
where $m_{\wedge\frakg^*}$ denotes the usual graded anticommutative product on $\bigwedge \frakg^*$.

A natural question at this point is whether the map $\Duf\otimes \id$ may be extended to an $A_\infty$-isomorphism in a universal way (i.e. independent of the specific choice of the Lie algebra $\frakg$). It has been answered in the negative by J. Alm ((\cite{Alm2011}, Remark 4.0.1), (\cite{Alm2014}, Proposition 5.3.0.10), see also J. Alm and S. Merkulov's paper \cite{AlmMerkulov}). The aim of this text is to describe a self-contained and elementary proof of this fact. The non-existence of such an $A_\infty$-isomorphism is equivalent to the following statement.

\begin{theorem}\label{thm:maintheoreom}
There does not exist a universal $A_\infty$-isomorphism of strict $A_\infty$-algebras,
\begin{equation*}
f:(S\frakg\otimes \bigwedge \frakg^*,(d_{S\frakg\otimes \wedge \frakg^*},\tilde m_0))\rightarrow (S\frakg\otimes \bigwedge \frakg^*,(d_{S\frakg\otimes \wedge \frakg^*},\tilde m_\Duf)),
\end{equation*}
whose first component $f_1$ is the identity.
\end{theorem}

\subsection{Kontsevich's product}
Let $x_1,\dots,x_d$ be a basis of $\frakg$ and denote by $c_{ij}^k\in \mathbb{K}$ its structure constants. Set $A:=S\frakg\otimes \bigwedge \frakg^*$. When equipped with the product $\tilde m_0$ we may identify $A$ with the polynomial algebra $\mathbb{K}[x_1,\dots,x_d,p^1,\dots, p^d]$, where $p^1,\dots,p^d$ are of degree 1 and describe the (degree shifted) dual basis. The $x_i$ are set to have degree 0. To describe the product $\tilde m_\Duf$ induced by the Duflo isomorphism on $A$, recall that in \cite{Kontsevich2003}, M. Kontsevich gave an explicit universal formula for the deformation quantization of any Poisson manifold. In particular, this can be applied to the dual Lie algebra $\frakg^*$ which defines a Poisson manifold with a linear Poisson structure. If we think of $x_1,\dots,x_d$ as local coordinates on $\frakg^*$, then for $f,g\in C^\infty (\frakg^*)$, the Poisson bracket is given by,
\begin{equation*}
\{f,g\}:=c^k_{ij}x_k\frac{\partial f}{\partial x_i}\frac{\partial g}{\partial x_j}.
\end{equation*}
It extends the linear Poisson structure on $\frakg^*$. We set $\pi_{ij}=c^k_{ij}x_k$ and 
\begin{equation*}
\pi:=\pi_{ij}p^i p^j=c^k_{ij}x_kp^i p^j\in A.
\end{equation*}
The latter can be identified with the Poisson bivector field on $\frakg^*$ given by $\pi_{ij}\frac{\partial}{\partial x_i}\wedge \frac{\partial}{\partial x_j}$ in local coordinates.
\begin{definition}
A deformation quantization of $\frakg^*$ is a $\mathbb{K}[[\epsilon]]$-linear, associative product $m_\pi$ on $C^\infty(\frakg^*) [[\epsilon]]$ such that for all $f,g\in C^\infty(\frakg^*)$,
\begin{equation*}
m_\pi(f\otimes g)=fg+\frac{\epsilon}{2}\{f,g\}+O(\epsilon^2).
\end{equation*}
\end{definition}

\begin{remark}
We shall not recall M. Kontsevich's construction at this point. Note however that (by formally setting $\epsilon=1$) it yields one further associative product $m_\pi$ on the space of polynomial functions on $\frakg^*$, i.e. on $S\frakg=\mathbb{K}[x_1,\dots,x_d]$. This product corresponds precisely to the one induced by the Duflo isomorphism (see \cite{Kontsevich2003},\cite{Shoiket2001}), that is,
\begin{equation*}
m_\pi=m_\Duf.
\end{equation*}
Moreover, if we define the product $\tilde m_\pi$ on $A$ in an analogous way as we did for $\tilde m_\Duf$, we also have $\tilde m_\pi=\tilde m_\Duf$.
\end{remark}

\section{A variant of M. Kontsevich's graph complex}

We consider the following version of M. Kontsevich's graph complex (\cite{Kontsevich2003}, Section 6.1). 

\begin{definition}
An \emph{admissible directed graph} is a directed graph $\Gamma$ with labeled vertices $1,2,\dots, n$ (called external), possibly other vertices (unlabeled and called internal) satisfying the following properties:
\begin{enumerate}\label{def:admissible}
\item{There is a linear order on the set of edges.}
\item{$\Gamma$ has no double edges, nor simple loops (edges connecting a vertex with itself).}
\item{Every internal vertex is at most trivalent.}
\item{Every internal vertex has at most one incoming edge, and at most two outgoing edges.}
\item{Every internal vertex can be connected by a path with an external vertex.}
\end{enumerate}
\end{definition}
Let $\dgr(n)$ be the vector space spanned by finite linear combinations of admissible directed graphs with $n$ external vertices, modulo the relation $\Gamma^\sigma=(-1)^{|\sigma|} \Gamma$, where $\Gamma^{\sigma}$ differs from $\Gamma$ by a permutation $\sigma$ on the order of edges. Here $|\sigma|$ denotes the parity of the permutation $\sigma$.

\begin {center}
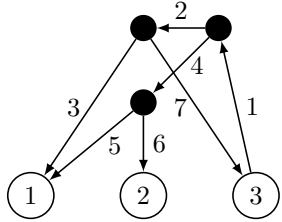
\begin{figure}[ht]
\resizebox{!}{3cm}{
\begin {tikzpicture}[-latex, auto ,node distance =1cm and 1cm ,on grid ,
semithick];

\tikzstyle{vertex1}=[circle,minimum size=2mm,draw=black,fill=black];
\tikzstyle{vertex2}=[circle,minimum size=2mm,draw=black,fill=white];
\tikzset{edge/.style = {-}};

\node[vertex2] (1) {$1$};
\node[vertex2] (2) [right=of 1,xshift=0.5cm] {$2$};
\node[vertex2] (3) [right=of 2,xshift=0.5cm] {$3$};

\node[vertex1] (4) [above=of 2,yshift=0.25cm] {};
\node[vertex1] (5) [above=of 4,xshift=0cm] {};
\node[vertex1] (6) [above=of 4,xshift=1cm] {};

\path (3) edge [] node[right] {$1$} (6);
\path (6) edge [] node[above] {$2$} (5);
\path (5) edge [] node[left] {$3$}(1);
\path (6) edge [] node[right] {$4$} (4);
\path (4) edge [] node[right,xshift=0.1cm] {$5$} (1);
\path (4) edge [] node[right] {$6$} (2);
\path (5) edge [] node[left] {$7$} (3);

\end{tikzpicture}
}
\caption{A graph in $\dgr(3)$.\label{fig:dgraphsexample}}
\end{figure}
\end{center}

\begin{definition}
\begin{enumerate}
\item{Let $\Gamma$ be an admissible directed graph and fix any one of its internal vertices. Call it $v$. Consider the linear combination obtained by replacing $v$ by two vertices connected by a directed edge $e$ and summing over all possible ways of reconnecting the edges previously adjacent to $v$ to the endpoints of $e$ while creating only admissible directed graphs.}
\item{Let $\Gamma$ be a directed graph as in Definition \ref{def:admissible} but with one four-valent internal vertex $v$ with one incoming and three outgoing edges. Consider the linear combination obtained by replacing $v$ by two vertices connected by a directed edge $e$ and summing over all possible ways of reconnecting the edges previously adjacent to $v$ to the endpoints of $e$ while creating only admissible directed graphs (see Figure \ref{fig:IHX}).}
\end{enumerate}
In both cases, the order of the set of edges of the new graphs is given by placing the newly added edge last. The \emph{generalized IHX relations} are given by setting such linear combinations equal to zero.
\end{definition}

\begin {center}
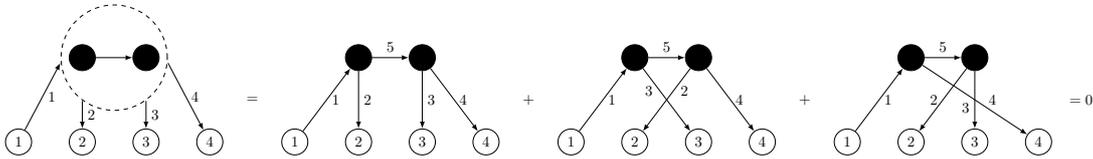
\begin{figure}[ht]
\resizebox{!}{2cm}{
\begin {tikzpicture}[-latex, auto ,node distance =1cm and 1cm ,on grid ,
semithick];

\tikzstyle{vertex1}=[circle,minimum size=2mm,draw=black,fill=black];
\tikzstyle{vertex2}=[circle,minimum size=2mm,draw=black,fill=white];
\tikzset{edge/.style = {-}};

\node[vertex2] (1) {$1$};
\node[vertex2] (2) [right=of 1,xshift=0.5cm] {$2$};
\node[vertex2] (3) [right=of 2,xshift=0.5cm] {$3$};
\node[vertex2] (4) [right=of 3,xshift=0.5cm] {$4$};

\node[vertex1] (5) [above=of 2,yshift=1cm] {$3$};
\node[vertex1] (6) [above=of 3,yshift=1cm] {$4$};

\draw[dashed] (2.25,2) circle (1.25cm);

\node (7) [left=of 5,xshift=0.55cm] {};
\node (8) [above=of 2,yshift=0.1cm] {};
\node (9) [above=of 3,yshift=0.1cm] {};
\node (10) [right=of 6,xshift=-0.55cm] {};

\path (1) edge [] node[right] {$1$} (7);
\path (8) edge [] node[right] {$2$} (2);
\path (9) edge [] node[right] {$3$} (3);
\path (10) edge [] node[right] {$4$} (4);

\path (5) edge [] node[above] {} (6);

\node (11) [right=of 4,yshift=1cm] {$=$};


\node[vertex2] (12) [right=of 11,yshift=-1cm] {$1$};
\node[vertex2] (13) [right=of 12,xshift=0.5cm] {$2$};
\node[vertex2] (14) [right=of 13,xshift=0.5cm] {$3$};
\node[vertex2] (15) [right=of 14,xshift=0.5cm] {$4$};

\node[vertex1] (16) [above=of 13,yshift=1cm] {$3$};
\node[vertex1] (17) [above=of 14,yshift=1cm] {$4$};

\path (12) edge [] node[right] {$1$} (16);
\path (16) edge [] node[right] {$2$} (13);
\path (17) edge [] node[right] {$3$} (14);
\path (17) edge [] node[right] {$4$} (15);

\path (16) edge [] node[above] {$5$} (17);

\node (18) [right=of 15,yshift=1cm] {$+$};


\node[vertex2] (19) [right=of 18,yshift=-1cm] {$1$};
\node[vertex2] (20) [right=of 19,xshift=0.5cm] {$2$};
\node[vertex2] (21) [right=of 20,xshift=0.5cm] {$3$};
\node[vertex2] (22) [right=of 21,xshift=0.5cm] {$4$};

\node[vertex1] (23) [above=of 20,yshift=1cm] {$3$};
\node[vertex1] (24) [above=of 21,yshift=1cm] {$4$};

\path (19) edge [] node[right] {$1$} (23);
\path (24) edge [] node[right,xshift=0.2cm,yshift=0.2cm] {$2$} (20);
\path (23) edge [] node[left,xshift=-0.2cm,yshift=0.2cm] {$3$} (21);
\path (24) edge [] node[right] {$4$} (22);

\path (23) edge [] node[above] {$5$} (24);

\node (25) [right=of 22,yshift=1cm] {$+$};


\node[vertex2] (26) [right=of 25,yshift=-1cm] {$1$};
\node[vertex2] (27) [right=of 26,xshift=0.5cm] {$2$};
\node[vertex2] (28) [right=of 27,xshift=0.5cm] {$3$};
\node[vertex2] (29) [right=of 28,xshift=0.5cm] {$4$};

\node[vertex1] (30) [above=of 27,yshift=1cm] {$3$};
\node[vertex1] (31) [above=of 28,yshift=1cm] {$4$};

\path (26) edge [] node[right] {$1$} (30);
\path (31) edge [] node[left] {$2$} (27);
\path (31) edge [] node[left,yshift=-0.2cm] {$3$} (28);
\path (30) edge [] node[right,xshift=0.2cm] {$4$} (29);

\path (30) edge [] node[above] {$5$} (31);

\node (32) [right=of 29,yshift=1cm] {$=0$};

\end{tikzpicture}
}
\caption{An example of the generalized IHX relations.\label{fig:IHX}}
\end{figure}
\end{center}

Our main object of study will be the collection of quotients ($n\geq 1$),
\begin{equation*}
\dgraphs(n):=\dgr(n)/\text{generalized IHX relations}.
\end{equation*}

\begin{remark}
For each $n\geq 1$, $\dgraphs(n)$ defines a graded vector space. The degree of a graph $\Gamma\in \dgraphs(n)$ is given by 
\begin{equation*}
|\Gamma|:=2\#\text{internal vertices}-\#\text{edges}.
\end{equation*}
Furthermore, the collection $\{\dgraphs(n)\}_{n\geq 1}$ assembles to a non-symmetric operad $\dgraphs$ in the category of graded vector spaces. The operadic composition in $\dgraphs$ is given by insertion. That is, for $\Gamma_1\in \dgraphs(r)$, $\Gamma_2\in \dgraphs(s)$, $1\leq j \leq r$,
\begin{equation*}
\Gamma_1 \circ_j \Gamma_2\in \dgraphs(r+s-1)
\end{equation*}
is constructed by replacing the $j$-th external vertex by $\Gamma_2$, summing over all possible ways of reconnecting the ``loose" edges (which were previously adjacent to vertex $j$) to vertices of $\Gamma_2$, and keeping only admissible directed graphs. The order on the set of edges of the new graphs is simply given by letting the edges of $\Gamma_1$ come before those of $\Gamma_2$ while leaving the respective orderings unchanged. Moreover, the product,
\begin{equation*}
\Gamma_1\circ\Gamma_2:=\sum\limits_{j=1}^r \Gamma_1\circ_j\Gamma_2,
\end{equation*}
defines a pre-Lie product (\cite{lodayvallette2012}, Section 5.9.15). Its graded commutator thus yields a graded Lie algebra structure on $\dgraphs$.
\end{remark}

\begin {center}
\begin{figure}[ht]
\resizebox{!}{2.25cm}{
\begin {tikzpicture}[-latex, auto ,node distance =1cm and 1cm ,on grid ,
semithick];

\tikzstyle{vertex1}=[circle,minimum size=2mm,draw=black,fill=black];
\tikzstyle{vertex2}=[circle,minimum size=2mm,draw=black,fill=white];
\tikzset{edge/.style = {-}};

\node[vertex2] (1) {$1$};
\node[vertex2] (2) [right=of 1,xshift=0.5cm] {$2$};
\path (1) edge [] node[above] {$1$}  (2);

\node (3) [right=of 2,xshift=-0.25cm,yshift=0.5cm] {$\circ_2$};

\node[vertex2] (4) [right=of 2,xshift=0.5cm] {$1$};
\node[vertex2] (5) [right=of 4]{$2$};

\node[vertex1] (6) [above=of 4,xshift=0.5cm,yshift=0.5cm] {};

\path (6) edge [] node[left] {$1$}  (4);
\path (6) edge [] node[right] {$2$}  (5);

\node (7) [right=of 5,xshift=0cm,yshift=0.5cm] {$=$};

\node[vertex2] (8) [right=of 7,xshift=0cm,yshift=-0.5cm] {$1$};
\node[vertex2] (9) [right=of 8,xshift=0.25cm] {$2$};
\node[vertex2] (10) [right=of 9,xshift=0.25cm] {$3$};

\node[vertex1] (19) [above=of 9,yshift=0.5cm] {};

\path (8) edge [] node[left] {$1$}  (19);
\path (19) edge [] node[left] {$2$}  (9);
\path (19) edge [] node[left] {$3$}  (10);


\node (11) [right=of 10,xshift=0cm,yshift=0.5cm] {$+$};

\node[vertex2] (12) [right=of 11,xshift=0cm,yshift=-0.5cm] {$1$};
\node[vertex2] (13) [right=of 12,xshift=0.25cm] {$2$};
\node[vertex2] (14) [right=of 13,xshift=0.25cm] {$3$};

\node[vertex1] (20) [above=of 13,yshift=0.5cm,xshift=0.5cm] {};

\path (12) edge [] node[above] {$1$}  (13);
\path (20) edge [] node[left] {$2$}  (13);
\path (20) edge [] node[right] {$3$}  (14);


\node (15) [right=of 14,xshift=0cm,yshift=0.5cm] {$+$};

\node[vertex2] (16) [right=of 15,xshift=0cm,yshift=-0.5cm] {$1$};
\node[vertex2] (17) [right=of 16,xshift=0.25cm] {$2$};
\node[vertex2] (18) [right=of 17,xshift=0.25cm] {$3$};

\node[vertex1] (21) [above=of 17,yshift=0.5cm,xshift=0.5cm] {};

\path (16) edge [bend right=40] node[below] {$1$}  (18);
\path (21) edge [] node[left] {$2$}  (17);
\path (21) edge [] node[right] {$3$}  (18);

\end{tikzpicture}
}
\caption{The operadic composition of two graphs in $\dgraphs(2)$.\label{fig:composition}}
\end{figure}
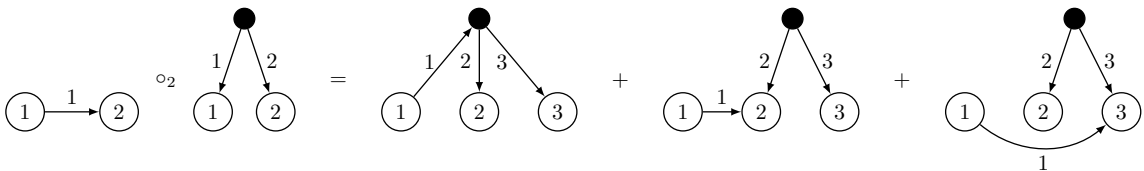
\end{center}

We will need the following two subspaces of $\dgraphs(n)$.

\begin{definition}
Let $\dgraphsuni(n)$ be the subspace of $\dgraphs(n)$ spanned by graphs for which all external vertices are univalent and denote by $\dgraphs(n,q)$ the subspace spanned by graphs with exactly $q$ edges connected to the $n$ external vertices.
\end{definition}

\begin{remark}
J. Alm and S. Merkulov (\cite{Alm2014},\cite{AlmMerkulov}) make use of a similar graph complex. They require, however, internal vertices to be \emph{at least trivalent}. For us, it will be crucial that the internal vertices are allowed to be uni- and bivalent.
\end{remark}

\begin{remark}
Undirected edges in our figures mean that we take the sum over all possible directions, i.e.

\begin {center}
\resizebox{!}{0.4cm}{
\begin {tikzpicture}[-latex, auto ,node distance =1cm and 1cm ,on grid ,
semithick];

\tikzstyle{vertex1}=[circle,minimum size=2mm,draw=black,fill=black];
\tikzstyle{vertex2}=[circle,minimum size=2mm,draw=black,fill=white];
\tikzset{edge/.style = {-}};

\node[vertex1] (1) {};
\node[vertex1] (2) [right=of 1] {};

\draw[edge] (1) to (2);

\node (3) [right=of 2] {$=$};

\node[vertex1] [right=of 3] (4) {};
\node[vertex1] (5) [right=of 4] {};

\node (6) [right=of 5] {$+$};

\node[vertex1] (7)[right=of 6]  {};
\node[vertex1] (8) [right=of 7] {};

\path (4) edge [] (5);
\path (8) edge [] (7);

\node (9) [right=of 8,xshift=-0.5cm] {$.$};

\end{tikzpicture}
}
\end{center}
\end{remark}

\section{M. Kontsevich's representation $\dgraphs\rightarrow \End_A$}
Next, we recall M. Kontsevich's construction (\cite{Kontsevich2003}, Section 6.3) of a linear map $B:\dgraphs\rightarrow \End_A$, $\Gamma\mapsto B_\Gamma$ following D. Calaque and C. A. Rossi's text \cite{RossiCalaque}. For this, let $\Gamma\in \dgr(n)$ and assume it has $m$ internal vertices. Label these $m$ internal vertices by $\bar{1},\dots, \bar{m}$ in an arbitrary way. Define the operator of degree one,
\begin{equation*}
\tau:=\sum\limits_{l=1}^d \frac{\partial}{\partial p^l}\otimes \frac{\partial}{\partial x_l}:A^{\otimes 2}\rightarrow A^{\otimes 2}.
\end{equation*}
Moreover, for any finite index set $I$ and any pair $(i,j)$ of distinct elements in $I$, let $\tau_{ij}:A^{\otimes I}\rightarrow A^{\otimes I}$ be the operator acting as $\tau$ on the $i$-th and $j$-th factors and as the identity on all other factors of $A$. Using this data, we define for any $n\geq 1$,
\begin{align*}
B'_\Gamma:A^{m+n}&\rightarrow A\\
B'_\Gamma(\gamma_1,\dots,\gamma_m,f_1,\dots,f_n)&:=\mu_{m+n}\prod\limits_{(i,j)\in E(\Gamma)}\tau_{ij}(\gamma_{\bar{1}}\otimes \cdots \otimes \gamma_{\bar{m}}\otimes f_1\otimes \cdots \otimes f_n)
\end{align*}
where $\mu_{m+n}:A^{\otimes (m+n)}\rightarrow A$ is the iterated graded commutative product,  $E(\Gamma)$ denotes the edge set of $\Gamma$, $(i,j)$ describes the edge starting at vertex $i$ and ending at $j$ for $i,j\in I:=\{1,\dots,n\}\cup\{\bar{1},\dots,\bar{m}\}$ and the order of the product of the endomorphisms $\tau_{ij}$ is determined by the order on the set of edges (i.e. the automorphism corresponding to the first edge in the linear order is applied first). 

Notice that since the automorphisms $\tau_{ij}$ are of degree one, any permutation $\sigma$ in the order of their product produces a sign $(-1)^{|\sigma |}$, where $|\sigma|$ denotes the parity of the permutation. This is compatible with the equivalence relation on $\dgr$ given by the ordering on the set of edges.

The map $B'_\Gamma$ depends on the choice of labeling of the internal vertices. However, the map,
\begin{align*}
B_\Gamma:A^{\otimes n}&\rightarrow A\\
B_\Gamma(f_1\otimes \dots \otimes f_n)&:=B'_\Gamma(\pi,\dots,\pi,f_1,\dots,f_n)
\end{align*}
is independent of the labeling we choose, and therefore yields a well-defined element of $\End_A(n)$.

\begin{proposition}\label{Prop:dgraphsmap}
The map $B:\dgr(n)\rightarrow \End_A(n)$ factors through the projection $\dgr(n)\twoheadrightarrow \dgraphs(n)$.
\end{proposition}

\begin{proof}
The generalized IHX relations are all obtained by replacing one internal vertex $v$ (which in this case is of valence less or equal to four) of some graph $\Gamma\in \dgr(n)$ by a directed edge connected by two internal vertices and summing over all possible ways of reconnecting the edges previously adjacent to $v$. If $v$ is for instance of valence four with one incoming and three outgoing edges, this will produce a term of the following form within the large product defined by $B_\Gamma(f_1\otimes \dots\otimes f_n)$ (for $f_1,\dots, f_n\in A$),
\begin{equation*}
\frac{\partial^4}{\partial x_i \partial p^j \partial p^k \partial p^l}\left(\frac{\partial \pi}{\partial x_m}\frac{\partial\pi}{\partial p^m}\right),
\end{equation*}
where the term $\frac{\partial \pi}{\partial x_m}\frac{\partial\pi}{\partial p^m}$ corresponds to the newly inserted directed edge. If the valence of $v$ is smaller the derivatives in front of this factor change accordingly. However, a short calculation shows that,
\begin{equation*}
\frac{\partial\pi}{\partial x_m}\frac{\partial \pi}{\partial p^m}=2\pi_{mi}\frac{\partial\pi_{jk}}{\partial x_m} p^i p^j p^k=\frac{2}{3}\left(\pi_{mi}\frac{\partial\pi_{jk}}{\partial x_m}+\pi_{mj}\frac{\partial\pi_{ki}}{\partial x_m}+\pi_{mk}\frac{\partial\pi_{ij}}{\partial x^m}\right) p^i p^j p^k=0.
\end{equation*}
This is zero since the term in the bracket is equivalent to the Jacobi identity in terms of the structure constants. Hence, the generalized IHX relations are sent to zero under the map $B:\dgr(n)\rightarrow \End_A(n)$ from which the statement follows.
\end{proof}

\begin{remark}
Proposition \ref{Prop:dgraphsmap} ensures that there is a well-defined map $\dgraphs\rightarrow \End_A$. We denote this map by $B$ as well. It follows from the product rule that this is a map of operads. More precisely, for $\Gamma_1\in \dgraphs(n)$, $\Gamma_2\in \dgraphs(m)$ and $f_1,\dots,f_{n+m-1}\in A$ we have,
\begin{align*}
B_{\Gamma_1\circ_i \Gamma_2}(f_1\otimes \cdots\otimes f_{n+m-1})=&B_{\Gamma_1}(f_1\otimes \dots\otimes f_{i-1} \otimes B_{\Gamma_2}(f_i\otimes \dots f_{i+m-1})\otimes f_{i+m}\otimes \dots \otimes f_{n+m-1})\\
=& B_{\Gamma_1}\circ_i B_{\Gamma_2} (f_1\otimes \dots \otimes f_{n+m-1}).
\end{align*}
\end{remark}

\begin{example}\label{example:example}
Consider $\Gamma_1\in \dgraphs(2)$ and $\Gamma_2\in \dgraphs(3)$ as in Figure \ref{fig:example1}. Then,
\begin{equation*}
B_{\Gamma_1}(f_1\otimes f_2)=\frac{\partial^2\pi}{\partial p^{i_1} \partial p^{i_2}}\frac{\partial^3\pi}{\partial x_{i_2}\partial p^{i_3}\partial p^{i_4}} \frac{\partial^2 f_1}{\partial x_{i_1}\partial x_{i_3}}\frac{\partial f_2}{\partial x_{i_4}}=c_{i_1,i_2}^j x_j c_{i_3,i_4}^{i_2}\frac{\partial^2 f_1}{\partial x_{i_1}\partial x_{i_3}}\frac{\partial f_2}{\partial x_{i_4}},
\end{equation*}
and,
\begin{equation*}
B_{\Gamma_2}(f_1\otimes f_2\otimes f_3)=\frac{\partial^3 \pi}{\partial x_{i_1} \partial p^{i_2} \partial p^{i_3}} \frac{\partial f_1}{\partial p^{i_1}}\frac{\partial f_2}{\partial x_{i_2}} \frac{\partial f_3}{\partial x_{i_3}}=c_{i_2,i_3}^{i_1} \frac{\partial f_1}{\partial p^{i_1}}\frac{\partial f_2}{\partial x_{i_2}} \frac{\partial f_3}{\partial x_{i_3}}.
\end{equation*}
\begin {center}
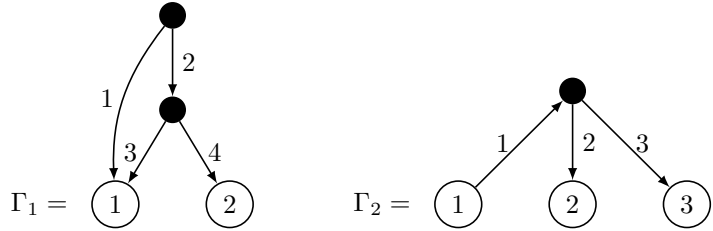
\begin{figure}[ht]
\resizebox{!}{3cm}{
\begin {tikzpicture}[-latex, auto ,node distance =1cm and 1cm ,on grid ,
semithick];

\tikzstyle{vertex1}=[circle,minimum size=2mm,draw=black,fill=black];
\tikzstyle{vertex2}=[circle,minimum size=2mm,draw=black,fill=white];
\tikzset{edge/.style = {-}};

\node[vertex2] (1) {$1$};
\node[vertex2] (2) [right=of 1,xshift=0.5cm] {$2$};

\node[vertex1] (3) [above=of 1,xshift=0.75cm,yshift=0.25cm] {};
\node[vertex1] (4) [above=of 3,yshift=0.25cm] {};

\path (3) edge [] node[left] {$3$} (1);
\path (3) edge [] node[right] {$4$} (2);
\path (4) edge [] node[right] {$2$}(3);
\path (4) edge [bend right=20] node[left] {$1$} (1);

\node[vertex2] (5) [right=of 2,xshift=2cm] {$1$};
\node[vertex2] (6) [right=of 5,xshift=0.5cm] {$2$};
\node[vertex2] (7) [right=of 6,xshift=0.5cm] {$3$};

\node[vertex1] (8) [above=of 6,yshift=0.5cm] {};

\path (5) edge [] node[left] {$1$} (8);
\path (8) edge [] node[right] {$2$} (6);
\path (8) edge [] node[right] {$3$} (7);

\node (9) [left=of 1,xshift=0cm] {$\Gamma_1=$};
\node (10) [left=of 5,xshift=0cm] {$\Gamma_2=$};

\end{tikzpicture}
}
\caption{The graphs corresponding to the calculations given in Example \ref{example:example}. \label{fig:example1}}
\end{figure}
\end{center}

\end{example}

\begin{remark}
For $n=2$, when the graph $\Gamma\in \dgraphs(2)$ has no edges starting at any external vertex and $f_1,f_2$ lie in $\mathbb{K}[x_1,\dots,x_d]\cong S\frakg$, the map $B_\Gamma$ corresponds precisely to M. Kontsevich's bidifferential operator $B_{\Gamma,\pi}$ (\cite{Kontsevich2003}, Section 2). Moreover, when $f_1,f_2\in A$, they decompose as $f_i=a_i(x_1,\dots,x_d)b_i(p^1,\dots,p^d)$ ($i=1,2$), and if $\Gamma\in \dgraphs(2)$ has no edges starting at any external vertex, $B_\Gamma$ satisfies,
\begin{equation*}
B_\Gamma(f_1,f_2)=B_\Gamma(a_1,a_2)b_1 b_2=B_{\Gamma,\pi}(a_1,a_2)b_1b_2.
\end{equation*}
One can verify (see \cite{Kontsevich2003}) that up to order $\epsilon^2$ (before setting $\epsilon=1$), M. Kontsevich's product $\tilde m_\pi=\tilde m_\Duf$ is given by,

\begin {center}
\resizebox{!}{2.5cm}{
\begin {tikzpicture}[-latex, auto ,node distance =1cm and 1cm ,on grid ,
semithick];

\tikzstyle{vertex1}=[circle,minimum size=2mm,draw=black,fill=black];
\tikzstyle{vertex2}=[circle,minimum size=2mm,draw=black,fill=white];
\tikzset{edge/.style = {-}};

\node(1)  {$\tilde m_\pi=$};

\node[vertex2] (3) [right=of 1,yshift=-1cm] {};
\node[vertex2] (4) [right=of 3] {};

\node (24) [right=of 4,xshift=-0.25cm,yshift=1cm] {$+\dfrac{1}{2}$};
\node[vertex2] (5)  [right=of 4,xshift=0.5cm] {};
\node[vertex2] (6) [right=of 5]  {};
\node[vertex1] (7)  [above=of 6, xshift=-0.5cm] {};

\path (7) edge [] node[left] {$1$} (5);
\path (7) edge [] node[right] {$2$} (6);

\node (25) [right=of 24,xshift=1.5cm] {$-\dfrac{1}{12}$};

\node[vertex2] (8)  [right=of 6,xshift=0.5cm] {};
\node[vertex2] (9) [right=of 8]  {};
\node[vertex1] (10)  [above=of 9, xshift=-0.5cm] {};
\node[vertex1] (11)  [above=of 10, xshift=0.5cm] {};

\path (10) edge [] node[left] {$3$} (8);
\path (10) edge [] node[left] {$4$} (9);
\path (11) edge [] node[left] {$2$} (10);
\path (11) edge [] node[right] {$1$} (9);

\node (26) [right=of 25,xshift=1.5cm] {$+\dfrac{1}{12}$};

\node[vertex2] (12)  [right=of 9,xshift=0.5cm] {};
\node[vertex2] (13) [right=of 12]  {};
\node[vertex1] (14)  [above=of 13, xshift=-0.5cm] {};
\node[vertex1] (15)  [above=of 14, xshift=-0.5cm] {};

\path (14) edge [] node[right] {$3$} (12);
\path (14) edge [] node[right] {$4$} (13);
\path (15) edge [] node[right] {$2$} (14);
\path (15) edge [] node[left] {$1$} (12);

\node (27) [right=of 26,xshift=1.5cm] {$+\dfrac{1}{8}$};

\node[vertex2] (16)  [right=of 13,xshift=0.5cm] {};
\node[vertex2] (17) [right=of 16]  {};
\node[vertex1] (18)  [above=of 16,yshift=1cm] {};
\node[vertex1] (19)  [above=of 17,yshift=1cm] {};

\path (18) edge [] node[left] {$1$} (16);
\path (18) edge [] node[left] {$2$} (17);
\path (19) edge [] node[right] {$3$} (16);
\path (19) edge [] node[right] {$4$} (17);

\node (28) [right=of 27,xshift=1.5cm] {$+\dfrac{1}{24}$};

\node[vertex2] (20)  [right=of 17,xshift=0.5cm] {};
\node[vertex2] (21) [right=of 20]  {};
\node[vertex1] (22)  [above=of 20,yshift=1cm] {};
\node[vertex1] (23)  [above=of 21,yshift=1cm] {};

\path (22) edge [] node[right] {$1$} (20);
\path (23) edge [] node[left] {$4$} (21);
\path (22) edge [bend left=50] node[above] {$2$} (23);
\path (23) edge [bend left=50] node[below] {$3$} (22);

\node (29) [right=of 28,xshift=1.5cm] {$+\cdots$};
\end{tikzpicture}
}
\end{center}

\end{remark}

\begin{example}
The Chevalley-Eilenberg differential $d_A$ on $A$ is represented by the graph,
\begin {center}
\resizebox{!}{1cm}{
\begin {tikzpicture}[-latex, auto ,node distance =1cm and 1cm ,on grid ,
semithick];

\tikzstyle{vertex1}=[circle,minimum size=2mm,draw=black,fill=black];
\tikzstyle{vertex2}=[circle,minimum size=2mm,draw=black,fill=white];
\tikzset{edge/.style = {-}};

\node (0) [left=of 1,yshift=0.5cm] {$a_1=$};

\node[vertex2] (1) {};
\node[vertex1] (2) [above=of 1] {};
\draw[edge] (1) to (2);

\node (3) [right=of 1,yshift=0.5cm] {$=$};

\node[vertex2] (4) [right=of 3,yshift=-0.5cm] {};
\node[vertex1] (5) [above=of 4] {};
\path (5) edge []  (4);

\node (6) [right=of 4,yshift=0.5cm] {$+$};

\node[vertex2] (7) [right=of 6,yshift=-0.5cm] {};
\node[vertex1] (8) [above=of 7] {};
\path (7) edge []  (8);

\node (9) [right=of 7,xshift=-0.4cm,yshift=0.4cm] {.};

\end{tikzpicture}
}
\end{center}
For $f\in A$, we find,
\begin{equation*}
d_A f= \frac{\partial \pi}{\partial p^l}\frac{\partial f}{\partial x_l}+\frac{\partial f}{\partial p^l}\frac{\partial \pi}{\partial x_l}=-2c_{il}^k x_k p^i \frac{\partial f}{\partial x_l}+\frac{\partial f}{\partial p^l} c_{ij}^l p^i p^j.
\end{equation*}
On basis elements it therefore acts as,
\begin{equation*}
d_A x_m=-2 c_{im}^k x_k p^i, \hspace{1cm} d_A p^m=c_{ij}^m p^i p^j.
\end{equation*}
Note that the usual convention is to define the Chevalley-Eilenberg differential as $-\frac{1}{2}d_A$.
\end{example}

\begin{definition}\label{def:universality}
A morphism $f:A^{\otimes n}\rightarrow A\in \End_A(n)$ is called \emph{universal} if there exists a graph $\Gamma\in \dgraphs(n)$ such that $f=B(\Gamma)\in \End_A(n)$. Any such morphism does not depend on the specific choice of the Lie algebra $\frakg$.
\end{definition}

\section{Proof of Theorem \ref{thm:maintheoreom}}

In Proposition \ref{prop:prelieisos} we have seen that as pre-Lie algebras,
\begin{equation*}
\frakg_{\As,A}\cong \prod\limits_{n\geq 1}\End_{sA}(n).
\end{equation*}
By Definition \ref{def:universality}, universal morphisms $f:(sA)^{\otimes n}\rightarrow sA$ are in bijection with homomorphisms $F:\KAs(n)\rightarrow \End_A(n)$ which factor through the map $B:\dgraphs(n)\rightarrow \End_A(n)$, i.e. homomorphisms for which there exists a morphism $\tilde F:\KAs(n)\rightarrow \dgraphs(n)$ making the diagram,
\begin {center}
\begin {tikzpicture}[-latex, auto ,node distance =1cm and 1cm ,on grid ,
semithick];
\tikzstyle{vertex1}=[circle,minimum size=2mm,draw=black,fill=black];
\tikzstyle{vertex2}=[circle,minimum size=2mm,draw=black,fill=white];
\tikzset{edge/.style = {-}};
\node (1) {$\KAs(n)$};
\node (2) [right=of 1,xshift=2cm] {$\End_A(n)$};
\node (3) [below=of 1,xshift=1.5cm,yshift=-0.25cm]{$\dgraphs(n)$};
\path (1) edge [] node[above] {$F$}  (2);
\path (3) edge [] node[right,xshift=0.2cm] {$B$} (2);
\path (1) edge [] node[left,xshift=-0.2cm] {$\tilde F$} (3);
\end{tikzpicture}
\end{center}
commute. Set,
\begin{equation*}
\frakg_{\As,\dgraphs}:=\Hom(\KAs,\dgraphs)=\prod\limits_{n\geq 1}\Hom(\KAs(n),\dgraphs(n)).
\end{equation*}
It forms a pre-Lie algebra, the product being given by the convolution product. A similiar identification as in Proposition \ref{prop:prelieisos} yields,
\begin{equation}\label{eq:graphs}
\frakg_{\As,\dgraphs}\cong \prod\limits_{n\geq 1} \dgraphs(n)\otimes \mathbb{K}{\mu_n^c}^*\cong \prod\limits_{n\geq 1} s^{-n+1}\dgraphs(n).
\end{equation}
For $\Gamma_1\in \dgraphs(n)$ and $\Gamma_2\in \dgraphs(m)$, we set,
\begin{equation*}
\Gamma_1\star\Gamma_2:=\sum\limits_{i=1}^n (-1)^{|\Gamma_2|(1-n)+(i+1)(m-1)}\Gamma_1\circ_i \Gamma_2.
\end{equation*}
This defines a pre-Lie bracket on $\prod_{n\geq 1}s^{-n+1}\dgraphs(n)$ which turns the bijections above into pre-Lie algebra isomorphisms.

\begin{lemma}\label{lemma:B}
Let $\tilde F, \tilde G\in \frakg_{\As,\dgraphs}$. Then,
$(B\circ \tilde F)\star (B\circ \tilde G)=B\circ (\tilde F\star \tilde G)$.
\end{lemma}

\begin{proof}
Since the respective pre-Lie products on $\frakg_{\As,A}$ and $\frakg_{\As,\dgraphs}$ are defined using the operadic composition, and $B$ respects all such operations, the statement follows.
\end{proof}

\begin{lemma}\label{lemma:universal}
There is a natural bijection,
\begin{equation*}
\MC(\frakg_{\As,\dgraphs})/\exp(\frakg^0_{\As,\dgraphs})\cong \{\text{universal } A_\infty\text{-structures on } A\}/\{\text{universal }A_\infty\text{-isomorphisms}\}.
\end{equation*}
\end{lemma}

\begin{proof}
A universal $A_\infty$-structure $m$ on $A$ corresponds to a Maurer-Cartan element $\alpha\in MC(\frakg_{\As,A})$ for which there exists an $\tilde \alpha\in \frakg_{\As,\dgraphs}$ such that $\alpha=B\circ \tilde \alpha$. By Lemma \ref{lemma:B} we have,
\begin{equation*}
0=\alpha\star \alpha=(B\circ \tilde \alpha)\star (B\circ \tilde \alpha)=B\circ (\tilde \alpha\star \tilde \alpha),
\end{equation*}
which is equivalent to $\tilde \alpha\star \tilde \alpha =0$ and $\tilde \alpha\in \MC(\frakg_{\As,\dgraphs})$. Moreover, universal $A_\infty$-isomorphisms correspond bijectively to elements of $\exp(\frakg^0_{\As,A})$ which factor through $B$. These may in turn be identified with elements of $\exp(\frakg^0_{\As,\dgraphs})$. Note that the discussion for $\frakg_{\As,A}$ from Section \ref{subsection:gaugeequivalence} may also be applied to $\frakg_{\As,\dgraphs}$ to define $\exp(\frakg^0_{\As,\dgraphs})$ and its action on $\MC(\frakg_{\As,\dgraphs})$.
\end{proof}

\begin{proposition}
Consider the Maurer-Cartan element $\alpha_\Duf\in \MC(\frakg_{\As,\dgraphs})$ corresponding to the universal $A_\infty$-structure $(d_A,\tilde m_\Duf,0,\dots)$. It is given by,

\begin {center}
\resizebox{!}{2.25cm}{
\begin {tikzpicture}[-latex, auto ,node distance =1cm and 1cm ,on grid ,
semithick];

\tikzstyle{vertex1}=[circle,minimum size=2mm,draw=black,fill=black];
\tikzstyle{vertex2}=[circle,minimum size=2mm,draw=black,fill=white];
\tikzset{edge/.style = {-}};

\node at (-1,0.5) {$\alpha_{\Duf}=$};

\node[vertex2] (1) {};
\node[vertex1] (2) [above=of 1] {};
\draw[edge] (1) to (2);

\node at (0.5,0.5) {$+$};

\node[vertex2] (3) [right=of 1] {};
\node[vertex2] (4) [right=of 3] {};

\node at (2.75,0.5) {$+\cfrac{1}{2}$};
\node[vertex2] (5)  [right=of 4,xshift=0.5cm] {};
\node[vertex2] (6) [right=of 5]  {};
\node[vertex1] (7)  [above=of 6, xshift=-0.5cm] {};

\path (7) edge [] node[left] {$1$} (5);
\path (7) edge [] node[right] {$2$} (6);

\node at (5.25,0.5) {$-\cfrac{1}{12}$};

\node[vertex2] (8)  [right=of 6,xshift=0.5cm] {};
\node[vertex2] (9) [right=of 8]  {};
\node[vertex1] (10)  [above=of 9, xshift=-0.5cm] {};
\node[vertex1] (11)  [above=of 10, xshift=0.5cm] {};

\path (10) edge [] node[left] {$3$} (8);
\path (10) edge [] node[left] {$4$} (9);
\path (11) edge [] node[left] {$2$} (10);
\path (11) edge [] node[right] {$1$} (9);

\node at (7.75,0.5) {$+\dfrac{1}{12}$};

\node[vertex2] (12)  [right=of 9,xshift=0.5cm] {};
\node[vertex2] (13) [right=of 12]  {};
\node[vertex1] (14)  [above=of 13, xshift=-0.5cm] {};
\node[vertex1] (15)  [above=of 14, xshift=-0.5cm] {};

\path (14) edge [] node[right] {$3$} (12);
\path (14) edge [] node[right] {$4$} (13);
\path (15) edge [] node[right] {$2$} (14);
\path (15) edge [] node[left] {$1$} (12);

\node at (10.25,0.5) {$+\dfrac{1}{8}$};

\node[vertex2] (16)  [right=of 13,xshift=0.5cm] {};
\node[vertex2] (17) [right=of 16]  {};
\node[vertex1] (18)  [above=of 16,yshift=1cm] {};
\node[vertex1] (19)  [above=of 17,yshift=1cm] {};

\path (18) edge [] node[left] {$1$} (16);
\path (18) edge [] node[left] {$2$} (17);
\path (19) edge [] node[right] {$3$} (16);
\path (19) edge [] node[right] {$4$} (17);

\node at (12.75,0.5) {$+\dfrac{1}{24}$};

\node[vertex2] (20)  [right=of 17,xshift=0.5cm] {};
\node[vertex2] (21) [right=of 20]  {};
\node[vertex1] (22)  [above=of 20,yshift=1cm] {};
\node[vertex1] (23)  [above=of 21,yshift=1cm] {};

\path (22) edge [] node[right] {$1$} (20);
\path (23) edge [] node[left] {$4$} (21);
\path (22) edge [bend left=50] node[above] {$2$} (23);
\path (23) edge [bend left=50] node[below] {$3$} (22);

\node at (15.25,0.5) {$+\cdots.$};
\end{tikzpicture}
}
\end{center}

\noindent It is gauge equivalent to the following Maurer-Cartan element,

\begin {center}
\resizebox{!}{1cm}{
\begin {tikzpicture}[-latex, auto ,node distance =1cm and 1cm ,on grid ,
semithick];

\tikzstyle{vertex1}=[circle,minimum size=2mm,draw=black,fill=black];
\tikzstyle{vertex2}=[circle,minimum size=2mm,draw=black,fill=white];
\tikzset{edge/.style = {-}};

\node (1) {$\alpha'_{\Duf}=$};

\node[vertex2] (2) [right=of 1,yshift=-0.5cm] {};
\node[vertex1] (3) [above=of 2] {};
\draw[edge] (2) to (3);

\node (4) [right=of 2,xshift=-0.5cm,yshift=0.5cm]  {$+$};

\node[vertex2] (5) [right=of 2] {};
\node[vertex2] (6) [right=of 5] {};

\node (7) [right=of 6,yshift=0.5cm]  {$+\cfrac{1}{24}$};

\node[vertex2] (8) [right=of 7,yshift=-0.5cm] {};
\node[vertex2] (9) [right=of 8] {};
\node[vertex2] (10) [right=of 9] {};
\node[vertex1] (11) [above=of 9] {};

\draw[edge] (11) to node[left] {$1$} (8);
\draw[edge] (11) to node[left] {$2$} (9);
\draw[edge] (11) to node[left] {$3$} (10);

\node (12) [right=of 10,yshift=0.5cm]  {$+\cdots.$};

\end{tikzpicture}
}
\end{center}

\noindent They are related via the gauge action of the element,

\begin {center}
\resizebox{!}{1cm}{
\begin {tikzpicture}[-latex, auto ,node distance =1cm and 1cm ,on grid ,
semithick];

\tikzstyle{vertex1}=[circle,minimum size=2mm,draw=black,fill=black];
\tikzstyle{vertex2}=[circle,minimum size=2mm,draw=black,fill=white];
\tikzset{edge/.style = {-}};

\node (1)  {$\xi=-\cfrac{1}{4}\big($};

\node[vertex2] (2) [right=of 1] {};
\node[vertex2] (3) [right=of 2] {};
\path (2) edge []  (3);

\node (4) [right=of 3, xshift=-0.5cm] {$-$};

\node[vertex2] (5) [right=of 4,xshift=-0.5cm] {};
\node[vertex2] (6) [right=of 5] {};
\path (6) edge []  (5);

\node (7) [right=of 6] {$\big)-\cfrac{1}{16}\big($};

\node[vertex2] (8) [right=of 7,yshift=-0.5cm] {};
\node[vertex2] (9) [right=of 8] {};
\node[vertex1] (10) [above=of 8, xshift=0.5cm] {};

\path (8) edge [] node[above] {$3$} (9);
\path (10) edge [] node[left] {$1$} (8);
\path (10) edge [] node[right] {$2$} (9);

\node (11) [right=of 9,yshift=0.5cm,xshift=-0.5cm] {$-$};

\node[vertex2] (12) [right=of 9] {};
\node[vertex2] (13) [right=of 12] {};
\node[vertex1] (14) [above=of 12, xshift=0.5cm] {};

\path (13) edge [] node[above] {$3$} (12);
\path (14) edge [] node[left] {$1$} (12);
\path (14) edge [] node[right] {$2$} (13);

\node (15) [right=of 13,yshift=0.5cm] {$\big)+\cfrac{1}{48}\big($};

\node[vertex2] (16) [right=of 15,yshift=-0.5cm] {};
\node[vertex2] (17) [right=of 16] {};
\node[vertex1] (18) [above=of 16, xshift=0.5cm] {};

\path (18) edge [] node[left] {$2$} (16);
\path (18) edge [bend left=30] node[right] {$3$} (17);
\path (17) edge [bend left=30] node[right] {$1$} (18);

\node (19) [right=of 17,yshift=0.5cm,xshift=-0.5cm] {$-$};

\node[vertex2] (20) [right=of 17] {};
\node[vertex2] (21) [right=of 20] {};
\node[vertex1] (22) [above=of 20, xshift=0.5cm] {};

\path (22) edge [] node[right] {$3$} (21);
\path (22) edge [bend right=30] node[left] {$2$} (20);
\path (20) edge [bend right=30] node[left] {$1$} (22);

\node (23) [right=of 21,yshift=0.5cm,xshift=0cm] {$\big)+\cdots.$};

\end{tikzpicture}
}
\end{center}
\end{proposition}

\begin{proof}
Rewrite $\alpha_\Duf$ as
\begin{equation*}
\alpha_\Duf=a_1+a_2+\frac{1}{2}a_3-\frac{1}{12} a_4+\frac{1}{8} a_5+\frac{1}{24} a_6+\cdots.
\end{equation*}
Note that $a_4$ consists of the difference of two graphs, while the other $a_i$ denote just one graph. Accordingly, we write,
\begin{equation*}
\xi=-\frac{1}{4}\xi_1-\frac{1}{16}\xi_2+\frac{1}{48}\xi_3+\cdots,
\end{equation*}
where now $\xi_i$ all correspond to the difference of two graphs as depicted in the proposition. It is easily verified that the terms contributing to $\alpha'_\Duf=e^{\ad_\xi} \alpha_\Duf$ up to a total number of four vertices are,
\begin{equation}\label{eq:mk}
[\alpha'_\Duf]=[\alpha_\Duf]-\frac{1}{4} \ad_{\xi_1}(a_1+a_2+\frac{1}{2}a_3)+\frac{1}{2!}\frac{1}{16}\ad^2_{\xi_1}(a_1+a_2)-\frac{1}{16} \ad_{\xi_2} (a_1+a_2)+\frac{1}{48}\ad_{\xi_3}(a_1+a_2),
\end{equation}
where $[\alpha_\Duf]$ ($[\alpha'_\Duf]$) denotes the part of $\alpha_\Duf$ (of $[\alpha'_\Duf]$) having up to four vertices. To fix notation, set,
\begin {center}
\resizebox{!}{1cm}{
\begin {tikzpicture}[-latex, auto ,node distance =1cm and 1cm ,on grid ,
semithick];

\tikzstyle{vertex1}=[circle,minimum size=2mm,draw=black,fill=black];
\tikzstyle{vertex2}=[circle,minimum size=2mm,draw=black,fill=white];
\tikzset{edge/.style = {-}};

\node (0) [left=of 1,yshift=0.5cm,xshift=0.5cm] {$b=$};

\node[vertex2] (1) {};
\node[vertex2] (2) [right=of 1] {};
\node[vertex2] (3) [right=of 2] {};
\node[vertex1] (4) [above=of 2] {};
\draw[edge] (4) to node[left] {$1$} (1);
\draw[edge] (4) to node[left] {$2$} (2);
\draw[edge] (4) to node[left] {$3$} (3);

\node (5) [right=of 3,yshift=0.5cm,xshift=-0.5cm] {$=$};

\node[vertex2] (6) [right=of 5,yshift=-0.5cm,xshift=-0.5cm] {};
\node[vertex2] (7) [right=of 6] {};
\node[vertex2] (8) [right=of 7] {};
\node[vertex1] (9) [above=of 7] {};
\path (6) edge [] node[left] {$1$} (9);
\path (9) edge [] node[left] {$2$} (7);
\path (9) edge [] node[left] {$3$} (8);

\node (10) [right=of 8,yshift=0.5cm,xshift=-0.5cm] {$+$};

\node[vertex2] (11) [right= of 10,yshift=-0.5cm,xshift=-0.5cm] {};
\node[vertex2] (12) [right=of 11] {};
\node[vertex2] (13) [right=of 12] {};
\node[vertex1] (14) [above=of 12] {};
\path (14) edge [] node[left] {$1$} (11);
\path (12) edge [] node[left] {$2$} (14);
\path (14) edge [] node[left] {$3$} (13);

\node (15) [right=of 13,yshift=0.5cm,xshift=-0.5cm] {$+$};

\node[vertex2] (16) [right=of 15,yshift=-0.5cm,xshift=-0.5cm] {};
\node[vertex2] (17) [right=of 16] {};
\node[vertex2] (18) [right=of 17] {};
\node[vertex1] (19) [above=of 17] {};
\path (19) edge [] node[left] {$1$} (16);
\path (19) edge [] node[left] {$2$} (17);
\path (18) edge [] node[left] {$3$} (19);

\node (20) [right=of 18,yshift=0.5cm,xshift=0.5cm] {$=: b_1+b_2+b_3,$};
\end{tikzpicture}
}
\end{center}

\noindent and

\begin {center}
\resizebox{!}{2.5cm}{
\begin {tikzpicture}[-latex, auto ,node distance =1cm and 1cm ,on grid ,
semithick];

\tikzstyle{vertex1}=[circle,minimum size=2mm,draw=black,fill=black];
\tikzstyle{vertex2}=[circle,minimum size=2mm,draw=black,fill=white];
\tikzset{edge/.style = {-}};

\node (0) [left=of 1,yshift=0.5cm,xshift=0.5cm] {$Q=-$};

\node[vertex2] (1) {};
\node[vertex2] (2) [right=of 1] {};
\node[vertex2] (3) [right=of 2] {};
\node[vertex1] (4) [above=of 2] {};
\path (1) edge [] node[left] {$1$} (4);
\path (4) edge [] node[left] {$2$} (2);
\path (4) edge [] node[left] {$3$} (3);

\node (5) [right=of 3,yshift=0.5cm,xshift=-0.5cm] {$-$};

\node[vertex2] (6) [right=of 5,yshift=-0.5cm,xshift=-0.5cm] {};
\node[vertex2] (7) [right=of 6] {};
\node[vertex2] (8) [right=of 7] {};
\node[vertex1] (9) [above=of 7] {};
\path (9) edge [] node[left] {$1$} (6);
\path (9) edge [] node[left] {$2$} (7);
\path (8) edge [] node[left] {$3$} (9);

\node (10) [right=of 8,yshift=0.5cm,xshift=-0.5cm] {$+$};

\node[vertex2] (11) [right= of 10,yshift=-0.5cm,xshift=-0.5cm] {};
\node[vertex2] (12) [right=of 11] {};
\node[vertex2] (13) [right=of 12] {};
\node[vertex1] (14) [above=of 12] {};
\path (14) edge [] node[left] {$1$} (11);
\path (11) edge [] node[below] {$3$} (12);
\path (14) edge [] node[left] {$2$} (13);

\node (15) [right=of 13,yshift=0.5cm,xshift=-0.5cm] {$-$};

\node[vertex2] (16) [right=of 15,yshift=-0.5cm,xshift=-0.5cm] {};
\node[vertex2] (17) [right=of 16] {};
\node[vertex2] (18) [right=of 17] {};
\node[vertex1] (19) [above=of 17] {};
\path (19) edge [] node[left] {$1$} (16);
\path (17) edge [] node[below] {$3$} (16);
\path (19) edge [] node[left] {$2$} (18);

\node (20) [right=of 18,yshift=0.5cm,xshift=-0.5cm] {$-$};

\node[vertex2] (21) [right=of 20,yshift=-0.5cm,xshift=-0.5cm] {};
\node[vertex2] (22) [right=of 21] {};
\node[vertex2] (23) [right=of 22] {};
\node[vertex1] (24) [above=of 22] {};
\path (24) edge [] node[left] {$1$} (21);
\path (22) edge [] node[below] {$3$} (23);
\path (24) edge [] node[left] {$2$} (23);

\node (25) [right=of 23,yshift=0.5cm,xshift=-0.5cm] {$+$};

\node[vertex2] (26) [right=of 25,yshift=-0.5cm,xshift=-0.5cm] {};
\node[vertex2] (27) [right=of 26] {};
\node[vertex2] (28) [right=of 27] {};
\node[vertex1] (29) [above=of 27] {};
\path (29) edge [] node[left] {$1$} (26);
\path (28) edge [] node[below] {$3$} (27);
\path (29) edge [] node[left] {$2$} (28);

\node (30) [below=of 0,yshift=-1cm] {$+$};

\node[vertex2] (31) [right=of 30,yshift=-0.5cm,xshift=-0.5cm] {};
\node[vertex2] (32) [right=of 31] {};
\node[vertex2] (33) [right=of 32] {};
\node[vertex1] (34) [above=of 32,xshift=-0.5cm] {};
\path (34) edge [] node[left] {$1$} (31);
\path (34) edge [] node[right] {$2$} (32);
\path (31) edge [bend right=30] node[below] {$3$} (33);

\node (35) [right=of 33,yshift=0.5cm,xshift=-0.5cm] {$-$};

\node[vertex2] (36) [right=of 35,yshift=-0.5cm,xshift=-0.5cm] {};
\node[vertex2] (37) [right=of 36] {};
\node[vertex2] (38) [right=of 37] {};
\node[vertex1] (39) [above=of 37,xshift=-0.5cm] {};
\path (39) edge [] node[left] {$1$} (36);
\path (39) edge [] node[right] {$2$} (37);
\path (38) edge [bend left=30] node[below] {$3$} (36);

\node (39) [right=of 38,yshift=0.5cm,xshift=-0.5cm] {$-$};

\node[vertex2] (40) [right=of 39,yshift=-0.5cm,xshift=-0.5cm] {};
\node[vertex2] (41) [right=of 40] {};
\node[vertex2] (42) [right=of 41] {};
\node[vertex1] (43) [above=of 42,xshift=-0.5cm] {};
\path (43) edge [] node[left] {$1$} (41);
\path (43) edge [] node[right] {$2$} (42);
\path (40) edge [bend right=30] node[below] {$3$} (42);

\node (44) [right=of 42,yshift=0.5cm,xshift=-0.5cm] {$+$};

\node[vertex2] (45) [right=of 44,yshift=-0.5cm,xshift=-0.5cm] {};
\node[vertex2] (46) [right=of 45] {};
\node[vertex2] (47) [right=of 46] {};
\node[vertex1] (48) [above=of 47,xshift=-0.5cm] {};
\path (48) edge [] node[left] {$1$} (46);
\path (48) edge [] node[right] {$2$} (47);
\path (47) edge [bend left=30] node[below] {$3$} (45);

\node (49) [right=of 47,xshift=-0.4cm,yshift=0.3cm] {.};

\end{tikzpicture}
}
\end{center}

The graphical calculus allows us to compute the following identities. Note that some graphs cancel because of symmetries or the generalized IHX relations.
\begin{align*}
[\xi_1, a_1]=&2 a_3,\\
[\xi_1,a_2]=&0,\\
[\xi_1,a_3]=&Q,\\
[\xi_1,[\xi_1,a_1]]=&2[\xi_1,a_3]=2Q,\\
[\xi_2,a_1]=&-a_4+2a_5,\\
[\xi_2,a_2]=&-[\xi_1,a_3]-b_1-b_3=-Q-b_1-b_3,\\
[\xi_3,a_1]=&-2 a_6+a_4,\\
[\xi_3,a_2]=&-b_1+2b_2-b_3.
\end{align*}
Inserting this back into equation \eqref{eq:mk} yields,
\begin{equation*}
[\alpha'_\Duf]=a_1+a_2+\frac{1}{24}(b_1+b_2+b_3)=a_1+a_2+\frac{1}{24}b.
\end{equation*}
\end{proof}

\begin{remark}
The Maurer-Cartan element $\alpha_0\in \MC(\frakg_{\As,\dgraphs})$ corresponding to the universal $A_\infty$-structure $(d_A,\tilde m_0,0,\dots)$ is given by $\alpha_0=a_1+a_2$.
\end{remark}

\subsection{The cohomology of $H(\frakg_{As,\dgraphs},\mathrm{ad}_{a_2})$}

Fix $n\geq 1$. Consider the polynomial coalgebra $P_n:=\mathbb{K}[t_1,\dots,t_n]$. The variables $t_i$ are of degree zero. It is equipped with the usual coproduct which on homogeneous elements is given by,
\begin{equation*}
\Delta(t_{i_1} \cdots  t_{i_k})=1\otimes t_{i_1} \cdots  t_{i_k} +\sum\limits_{j=1}^{k-1} t_{i_1} \cdots  t_{i_j}\otimes t_{i_{j+1}}\cdots t_{i_k}+t_{i_1} \cdots  t_{i_k}\otimes 1.
\end{equation*}
Let $\Omega P_n=(T s^{-1} P_n,d)$ be its cobar construction. Here, $d:(s^{-1}P_n)^{\otimes k} \rightarrow (s^{-1}P_n)^{\otimes k+1}$ is the degree one map explicitly given by the alternating sum,
\begin{equation*}
d=\sum\limits_{i=0}^{k+1} (-1)^i d_i,
\end{equation*}
where for $p\in (s^{-1}P_n)^{\otimes k} $, $d_0 (p)=s^{-1}\otimes p$, $d_{k+1}(p)=p\otimes s^{-1}$ and for $p_i\in s^{-1}P_n$,
\begin{equation*}
d_i(p_1\otimes \cdots \otimes p_k)=p_1\otimes \cdots \otimes \Delta (p_i)\otimes \cdots \otimes p_k.
\end{equation*}
On $\Omega P_n$ there is an $\mathbb{N}_0^n$-grading counting the number of $t_1, \dots, t_n$ appearing in any monomial. The subspace of degree $(1,\dots,1)$ elements (i.e. every variable occurs exactly once), denoted by $\Omega P_n^{(1,\dots,1)}$, defines a subcomplex of $\Omega P_n$. For instance, (up to suspension) $t_1\otimes t_2t_3$ lies in  $\Omega P_3^{(1,1,1)}$, whereas $t_1\otimes t_2t_1$ does not. Note also that the $S_n$-action on $\Omega P_n^{(1,\dots,1)}$ which permutes the variables is compatible with the differential. 

\begin{proposition}
The cohomology of $(\Omega P_n^{(1,\dots,1)},d)$ is one-dimensional. More precisely, if we denote by,
\begin{equation*}
\omega_n=\frac{1}{n!}\sum\limits_{\sigma\in S_n}(-1)^{|\sigma|}s^{-1}t_{\sigma(1)}\otimes \cdots \otimes s^{-1}t_{\sigma(n)},
\end{equation*}
the totally antisymmetric element in $(s^{-1}P_n)^{\otimes n}$, then,
\begin{equation*}
H(\Omega P_n^{(1,\dots,1)},d)=H^n(\Omega P_n^{(1,\dots,1)},d)=\mathbb{K}\cdot [\omega_n]\cong s^{-n}\mathbb{K}.
\end{equation*}
\end{proposition}

\begin{proof}
The cohomology of the complex $(\Omega P_n^{(1,\dots,1)},d)$ was computed for instance in (\cite{Drinfeld89}, Proposition 2.2) by V. Drinfeld, in (\cite{Dror1997}, Section 4.2) by D. Bar-Natan and in \cite{Cubicalcomplex} by P. \v Severa and T. Willwacher.
\end{proof}

\begin{lemma}\label{lemma:complexes}
Consider $\dgraphsuni(n)$ equipped with the $S_n$-action which permutes the labels of the external vertices. There is an isomorphism of complexes,
\begin{equation*}
\left(s \hspace{0.05cm} \dgraphsuni(n)\otimes_{S_n}\Omega P_n^{(1,\dots,1)},1\otimes d \right)\cong (\prod\limits_{m\geq 1}s^{-m+1}\dgraphs(m, n),\ad_{a_2}).
\end{equation*}
\end{lemma}

A graphical interpretation of the isomorphism is given in Figure \ref{fig:S4isomorphism}. In the following, we denote by $s^{-\bullet+1}\dgraphs(\bullet,n):=\prod_{m\geq 1}s^{-m+1}\dgraphs(m,n)$.

\begin{proof}
The identification as described in Figure \ref{fig:S4isomorphism} is bijective. Moreover, it is easy to check that the differentials $1\otimes d$ and $\ad_{a_2}$ act in an equivalent way on their respective complexes.
\end{proof}

\begin {center}
\begin{figure}[ht]
\resizebox{!}{3cm}{
\begin {tikzpicture}[-latex, auto ,node distance =1cm and 1cm ,on grid ,
semithick];

\tikzstyle{vertex1}=[circle,minimum size=2mm,draw=black,fill=black];
\tikzstyle{vertex2}=[circle,minimum size=2mm,draw=black,fill=white];
\tikzstyle{vertex3}=[circle,minimum size=2mm,draw=black,dashed,fill=white];
\tikzset{edge/.style = {-}};

\node[vertex2] (1) {$1$};
\node[vertex2] (2) [right=of 1] {$3$};
\node[vertex2] (3) [right=of 2] {$2$};
\node[vertex2] (4) [right=of 3] {$4$};

\node (5) [above=of 1,yshift=0.13cm] {};
\node (6) [above=of 2,yshift=0.13cm] {};
\node (7) [above=of 3,yshift=0.13cm] {};
\node (8) [above=of 4,yshift=0.13cm] {};

\draw[edge] (1) to (5);
\draw[edge] (2) to (6);
\draw[edge] (3) to (7);
\draw[edge] (4) to (8);

\draw[edge] (-1,1) to (4,1);
\draw[edge] (-1,1) to (-1,2);
\draw[edge] (-1,2) to (4,2);
\draw[edge] (4,2) to (4,1);



\node (10) [right=of 4,xshift=1cm,yshift=1.cm] {$\otimes_{S_4}$};


\node[vertex2] (11) [right=of 4,xshift=2cm] {$1$};
\node[vertex2] (12) [right=of 11,xshift=0.5cm] {$2$};
\node[vertex2] (13) [right=of 12,xshift=0.5cm] {$3$};

\node (14) [above=of 11,yshift=0.5cm] {$s^{-1}t_1$};
\node (15) [above=of 12,yshift=0.5cm] {$s^{-1}t_4 t_2$};
\node (16) [above=of 13,yshift=0.5cm] {$s^{-1}t_3$};

\node (17) [above=of 11,xshift=0.75cm, yshift=0.5cm] {$\otimes$};
\node (18) [above=of 12,xshift=0.75cm, yshift=0.5cm] {$\otimes$};

\draw[edge] (11) to (14);
\draw[edge] [bend left=40] (12) to (15);
\draw[edge] [bend right= 40] (12) to (15);
\draw[edge] (13) to (16);

\node (19) [right=of 13,xshift=0cm,yshift=1cm] {$\mapsto$};

\node[vertex2] (20) [right=of 13,xshift=1cm] {$1$};
\node[vertex2] (21) [right=of 20,xshift=0.5cm] {$2$};
\node[vertex2] (22) [right=of 21,xshift=0.5cm] {$3$};

\node[vertex3] (23) [above=of 20,yshift=0.5cm] {$t_1$};
\node[vertex3] (24) [above=of 21,yshift=0.5cm,xshift=-0.5cm] {$t_4$};
\node[vertex3] (25) [above=of 21,yshift=0.5cm,xshift=0.5cm] {$t_2$};
\node[vertex3] (26) [above=of 22,yshift=0.5cm] {$t_3$};

\draw[edge] (20) to (23);
\draw[edge] [bend left=40] (21) to (24);
\draw[edge] [bend right= 40] (21) to (25);
\draw[edge] (22) to (26);

\node (27) [above=of 23,yshift=0.63cm] {};
\node (28) [above=of 24,yshift=0.57cm] {};
\node (29) [above=of 25,yshift=0.63cm] {};
\node (30) [above=of 26,yshift=0.57cm] {};

\draw[edge] (27) to (23);
\draw[edge] [bend right=10] (30) to (24);
\draw[edge] [bend right=30] (29) to (25);
\draw[edge] [bend left=10] (28) to (26);

\draw[edge] (10,3) to (15,3);
\draw[edge] (10,4) to (15,4);
\draw[edge] (10,3) to (10,4);
\draw[edge] (15,3) to (15,4);


\node (32) [right=of 22,yshift=1cm,xshift=0.5cm] {$=$};

\node[vertex2] (33) [right=of 22,xshift=2cm] {$1$};
\node[vertex2] (34) [right=of 33,xshift=0.5cm] {$2$};
\node[vertex2] (35) [right=of 34,xshift=0.5cm] {$3$};

\draw[edge] (16,2) to (21,2);
\draw[edge] (16,3) to (21,3);
\draw[edge] (16,2) to (16,3);
\draw[edge] (21,2) to (21,3);

\node (36) [above=of 33,yshift=1.13cm] {};
\node (37) [above=of 34,yshift=1.13cm,xshift=-0.5cm] {};
\node (38) [above=of 34,yshift=1.13cm,xshift=0.5cm] {};
\node (39) [above=of 35,yshift=1.13cm] {};

\draw[edge] (33) to (36);
\draw[edge] (34) to (38);
\draw[edge] (34) to (39);
\draw[edge] (35) to (37);


\end{tikzpicture}
}
\caption{Schematic description of the isomorphism from Lemma \ref{lemma:complexes}. The element $\Gamma\otimes_{S_4}(s^{-1}t_1\otimes s^{-1}t_4t_2\otimes s^{-1}t_3)\in \dgraphsuni(4)\otimes_{S_4} \Omega P_4^{(1,\dots,1)}$ is identified by this gluing procedure with a graph in $\dgraphs(3)$.\label{fig:S4isomorphism}}
\end{figure}
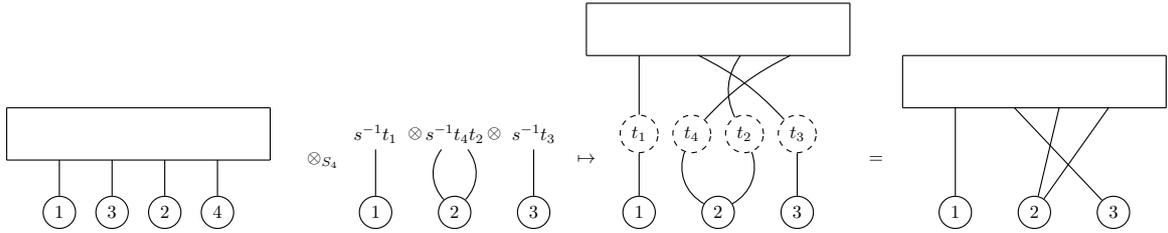
\end{center}

\begin{corollary}\label{cor:cohomdefcomplex}
We have,
\begin{equation*}
H(s\hspace{0.05cm} \dgraphsuni(n) \otimes_{S_n}\Omega P_n^{(1,\dots,1)},1\otimes d)=s \hspace{0.05cm} \dgraphsuni(n)\otimes_{S_n} \mathbb{K}\cdot [\omega_n].
\end{equation*}
Thus, $H(s^{-\bullet+1}\dgraphs(\bullet,n),\ad_{a_2})$ is spanned by graphs with univalent external vertices which are totally antisymmetric with respect to the $S_n$-action permuting the external vertices. Moreover, by taking the direct product over all $n$, we obtain,
\begin{equation*}
H(\frakg_{\As,\dgraphs},\ad_{a_2})\cong\prod\limits_{n\geq 1} s\hspace{0.05cm} \dgraphsuni(n)\otimes_{S_n}\mathbb{K}\cdot [\omega].
\end{equation*}
\end{corollary}

\begin{proof}
The first part of the statement follows from the fact that taking coinvariants under finite group actions commutes with taking cohomology. The rest is a consequence of Lemma \ref{lemma:complexes}.
\end{proof}

\begin{lemma}
The graph $b\in \dgraphs(3)$ represents a non-trivial class in $H^1(\frakg_{\As,\dgraphs},\ad_{a_1+a_2})$ under the identification \eqref{eq:graphs}.
\end{lemma}

\begin{proof}
By Corollary \ref{cor:cohomdefcomplex}, we know that $b$ represents a non-trivial class in $H^1(\frakg_{\As,\dgraphs},\ad_{a_2})$. It is also easily verified that $\ad_{a_1}(b)=0$. Moreover, $b$ is cohomologous to the graph $b'$ on the left in Figure \ref{fig:cohomologous}. They satisfy $b=-b'+(\ad_{a_1+a_2})(c)$, where $c$ is the graph on the right in Figure \ref{fig:cohomologous}. Since $b'$ has no internal vertices it cannot be exact under $\ad_{a_1+a_2}$, as all graphs in the image of the differential $\ad_{a_1}$ have at least on internal vertex. Thus, $b$ is not exact and the statement follows.
\end{proof}

\begin {center}
\begin{figure}[ht]
\resizebox{!}{1.5cm}{
\begin {tikzpicture}[-latex, auto ,node distance =1cm and 1cm ,on grid ,
semithick];

\tikzstyle{vertex1}=[circle,minimum size=2mm,draw=black,fill=black];
\tikzstyle{vertex2}=[circle,minimum size=2mm,draw=black,fill=white];
\tikzset{edge/.style = {-}};

\node (1)  {$b'=$};

\node[vertex2] (2) [right=of 1] {$1$};
\node[vertex2] (3) [right=of 2,xshift=0.5cm] {$2$};
\node[vertex2] (4) [right=of 3,xshift=0.5cm] {$3$};
\node[vertex2] (5) [right=of 4,xshift=0.5cm] {$4$};

\draw[edge] (2) to  node[above] {$1$}  (3);
\draw[edge] (4) to  node[above] {$2$}  (5);

\node (6) [right=of 5] {$+$};

\node[vertex2] (7) [right=of 6] {$1$};
\node[vertex2] (8) [right=of 7,xshift=0.5cm] {$2$};
\node[vertex2] (9) [right=of 8,xshift=0.5cm] {$3$};
\node[vertex2] (10) [right=of 9,xshift=0.5cm] {$4$};

\draw[edge] (7) [bend left=40] to  node[above] {$1$}(9);
\draw[edge] (8) [bend right=40] to  node[below] {$2$}(10);

\node (11) [right=of 10,xshift=0.5cm] {$c=$};

\node[vertex2] (12) [right=of 11] {$1$};
\node[vertex2] (13) [right=of 12,xshift=0.5cm] {$2$};
\node[vertex2] (14) [right=of 13,xshift=0.5cm] {$3$};

\draw[edge] (12) to  node[above] {$1$} (13);
\draw[edge] (13) to  node[above] {$2$} (14);

\node (15) [right=of 10,xshift=-0.4cm,yshift=-0.2cm] {,};

\end{tikzpicture}
}
\caption{We have $b=-b'+[a_1+a_2,c]$.\label{fig:cohomologous}}
\end{figure}
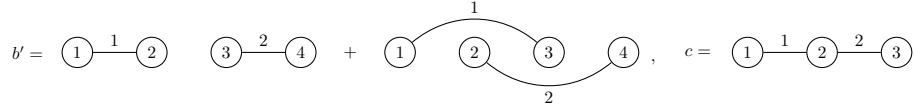
\end{center}

\begin{corollary}
The Maurer-Cartan elements $\alpha_\Duf$ and $\alpha_0$ are not gauge equivalent.
\end{corollary}

\begin{proof}
Define a second grading on $\frakg_{\As,\dgraphs}\cong \prod_{n\geq 1}s^{-n+1}\dgraphs(n)$ by
\begin{equation*}
(\#\text{internal vertices}+\#\text{external vertices})-1.
\end{equation*}
The grading is compatible with the Lie algebra structure. The degree one part of $\alpha'_\Duf$, denoted by $\alpha'^{(1)}_\Duf$, is given by $\alpha_0=a_1+a_2$. The degree three part $\alpha'^{(3)}_\Duf$ equals $1/24\cdot b$. Since $ [b]\neq 0 \in H^1(\frakg_{\As,\dgraphs},\ad_{\alpha_0})$, we may apply Lemma \ref{lemma:technical} to find that $\alpha_0$ is not gauge equivalent to $\alpha'_\Duf$. The fact that $\alpha_\Duf$ is gauge equivalent to $\alpha'_\Duf$ establishes the result.
\end{proof}

\begin{proof}[Proof of Theorem \ref{thm:maintheoreom}]
The universal Maurer-Cartan elements $\alpha_\Duf$ and $\alpha_0$ correspond to the universal $A_\infty$-structures $(d_A,\tilde m_\Duf,0,\dots)$ and $(d_A,\tilde m_0,0,\dots)$ on $A$. Since $\alpha_\Duf$ is not gauge equivalent to $\alpha_0$, there is no universal $A_\infty$-isomorphism between these two structures, by Lemma \ref{lemma:universal}.
\end{proof}

\bibliographystyle{plain}
\bibliography{Infinityduflo}

\begin{thebibliography}{10}

\bibitem{Alm2011}
J.~{Alm}.
\newblock {Two-colored noncommutative Gerstenhaber formality and infinity Duflo
  isomorphism}.
\newblock {\em arXiv:1104.2194 [math.QA]}, 2011.

\bibitem{Alm2014}
J.~Alm.
\newblock {\em Universal algebraic structures on polyvector fields}.
\newblock PhD thesis, Stockholm University, 2014.

\bibitem{AlmMerkulov}
J.~Alm and S.~Merkulov.
\newblock Grothendieck-{T}eichm\"uller group and {P}oisson cohomologies.
\newblock {\em J. Noncommut. Geom.}, 9(1):185--214, 2015.

\bibitem{Dror1997}
D.~Bar-Natan.
\newblock Non-associative tangles.
\newblock In {\em Geometric topology ({A}thens, {GA}, 1993)}, volume~2 of {\em
  AMS/IP Stud. Adv. Math.}, pages 139--183. Amer. Math. Soc., Providence, RI,
  1997.

\bibitem{Dolgushev2012}
H.~Bursztyn, V.~Dolgushev, and S.~Waldmann.
\newblock Morita equivalence and characteristic classes of star products.
\newblock {\em J. Reine Angew. Math.}, 662:95--163, 2012.

\bibitem{RossiCalaque}
D.~Calaque and C.~A. Rossi.
\newblock {\em Lectures on {D}uflo isomorphisms in {L}ie algebra and complex
  geometry}.
\newblock EMS Series of Lectures in Mathematics. European Mathematical Society
  (EMS), Z\"urich, 2011.

\bibitem{Chevalley1948}
C.~Chevalley and S.~Eilenberg.
\newblock Cohomology theory of {L}ie groups and {L}ie algebras.
\newblock {\em Trans. Amer. Math. Soc.}, 63:85--124, 1948.

\bibitem{Drinfeld89}
V.~G. Drinfeld.
\newblock Quasi-{H}opf algebras.
\newblock {\em Algebra i Analiz}, 1(6):114--148, 1989.

\bibitem{Duflo77}
M.~Duflo.
\newblock Op\'erateurs diff\'erentiels bi-invariants sur un groupe de {L}ie.
\newblock {\em Ann. Sci. \'Ecole Norm. Sup. (4)}, 10(2):265--288, 1977.

\bibitem{Goldman1988}
W.~M. Goldman and J.~J. Millson.
\newblock The deformation theory of representations of fundamental groups of
  compact {K}{\"a}hler manifolds.
\newblock {\em Publications Math{\'e}matiques de l'Institut des Hautes
  {\'E}tudes Scientifiques}, 67(1):43--96, Jan 1988.

\bibitem{Keller2001}
B.~Keller.
\newblock Introduction to {$A$}-infinity algebras and modules.
\newblock {\em Homology Homotopy Appl.}, 3(1):1--35, 2001.

\bibitem{Kontsevich2003}
M.~Kontsevich.
\newblock Deformation quantization of {P}oisson manifolds.
\newblock {\em Lett. Math. Phys.}, 66(3):157--216, 2003.

\bibitem{lodayvallette2012}
J.L. Loday and B.~Vallette.
\newblock {\em Algebraic Operads}.
\newblock Grundlehren der mathematischen Wissenschaften. Springer Berlin
  Heidelberg, 2012.

\bibitem{PevznerTorossian2004}
M.~Pevzner and Ch. Torossian.
\newblock Isomorphisme de {D}uflo et la cohomologie tangentielle.
\newblock {\em J. Geom. Phys.}, 51(4):487--506, 2004.

\bibitem{Cubicalcomplex}
P.~{Severa} and T.~{Willwacher}.
\newblock {The cubical complex of a permutation group representation - or
  however you want to call it}.
\newblock {\em arXiv:1103.3283v2 [math.QA]}, March 2011.

\bibitem{Shoiket2001}
B.~Shoikhet.
\newblock Vanishing of the {K}ontsevich integrals of the wheels.
\newblock {\em Lett. Math. Phys.}, 56(2):141--149, 2001.
\newblock EuroConf\'erence Mosh\'e Flato 2000, Part II (Dijon).

\bibitem{Shoikhet2003}
B.~Shoikhet.
\newblock Tsygan formality and {D}uflo formula.
\newblock {\em Math. Res. Lett.}, 10(5-6):763--775, 2003.

\bibitem{Willwacher2014}
T.~Willwacher.
\newblock M. {K}ontsevich's graph complex and the
  {G}rothendieck--{T}eichm{\"u}ller {L}ie algebra.
\newblock {\em Inventiones mathematicae}, 200(3):671--760, 2014.

\end{thebibliography}

\end{document}